\documentclass[12pt]{article}
\usepackage{amsmath, amsthm, amssymb}
\usepackage[pdftex]{graphicx}
\usepackage{enumerate}
\usepackage{mathrsfs}
\usepackage{color}
\def\XXint#1#2#3{{\setbox0=\hbox{$#1{#2#3}{\int}$}
    \vcenter{\hbox{$#2#3$}}\kern-.5\wd0}}

\def\({\left(}
\def\){\right)}

\def\RR{{\mathbb{R}}}
\def\CC{{\mathbb{C}}}

\newcommand{\Om}{\Omega}

\newcommand{\eia}{e^{i\int_{z_{n-1}}^{z_n} A_z(x,y,z)\, dz}}

\newcommand{\eps}{\epsilon}

\newcommand\lep{|\ln \eps|}

\newcommand{\pP}{\psi_{+}}
\newcommand{\pM}{\psi_{-}}
\newcommand{\fP}{f_{+}}
\newcommand{\fM}{f_{-}}
\newcommand{\tta}{\theta}
\newcommand{\nP}{n_{+}}
\newcommand{\nM}{n_{-}}

\newcommand{\dis}{\displaystyle}

\newcommand{\be}{\begin{equation}}
\newcommand{\ee}{\end{equation}}
\newcommand{\bea}{\begin{eqnarray}}
\newcommand{\eea}{\end{eqnarray}}
\newcommand{\beann}{\begin{eqnarray*}}
\newcommand{\eeann}{\end{eqnarray*}}
\newcommand{\nnn}{\nonumber}

 \newtheorem{theorem}{Theorem}[section]

\newtheorem{lemma}[theorem]{Lemma}

\newtheorem{prop}[theorem]{Proposition}

\begin{document}
\title{Symmetric vortices for two-component Ginzburg--Landau systems}
\author{{\Large Stan Alama}\footnote{Dept. of Mathematics and Statistics,
McMaster Univ., Hamilton, Ontario, Canada L8S 4K1.  Supported
by an NSERC Research Grant.} \and  {\Large Qi Gao${}^*$}
 }

\thispagestyle{empty}
\maketitle

\begin{abstract}
We study Ginzburg--Landau equations for a complex vector order parameter
$\Psi=(\psi_+,\psi_-)\in\mathbb C^2$.  We consider symmetric vortex solutions in the plane $\RR^2$, $\psi(x)=f_\pm(r)e^{in_\pm\theta}$, with given degrees $n_\pm\in\mathbb Z$, and prove existence, uniqueness, and asymptotic behavior of solutions as $r\to\infty$.  We also consider the monotonicity properties of solutions, and exhibit parameter ranges in which both vortex profiles $f_+,f_-$ are monotone, as well as parameter regimes where one component is non-monotone.  The qualitative results are obtained by means of a sub- and supersolution construction and a comparison theorem for elliptic systems.

\bigskip


\medskip


\end{abstract}

\newpage


\baselineskip=18pt

\section{Introduction}

In this paper we study entire solutions in $\RR^2$ of coupled systems of two Ginzburg-Landau type equations, for complex vector-valued functions 
$\Psi=(\psi_+(x),\psi_-(x)): \ \RR^2\to \mathbb C^2$,
\begin{equation}\label{eqns}
\left.
\begin{gathered}
-\Delta\psi_{+}+[A_{+}(|\psi_{+}|^{2}-t^2_{+})+B(|\psi_{-}|^{2}-t^2_{-})]\psi_{+}=0,\\ 
-\Delta\psi_{-}+[A_{-}(|\psi_{-}|^{2}-t^2_{-})+B(|\psi_{+}|^{2}-t^2_{+})]\psi_{-}=0.
\end{gathered}
\right\}
\end{equation}
Coupled systems of two Ginzburg-Landau equations of the type studied here arise both in the modeling of $p$-wave superconductors \cite{KnRo} and in two-component Bose-Einstein condensates (BEC) \cite{Eto}.  We give a brief discussion of the connection between the energy \eqref{energy1} and the two component BEC model at the end of this section.

Throughout the paper we make the following assumptions concerning the constants appearing in \eqref{eqns}:
\begin{equation*}  
A_+, A_->0, \ B^2<A_+A_-, \quad t_+, t_->0.
 \tag{\text{H} } 
 \end{equation*}
 
The system \eqref{eqns} gives the Euler-Lagrange equations formally associated to the energy 
 \begin{equation}\label{energy1}
E(\Psi; \Omega)=\dis\int_{\Omega}\frac{1}{2}|\nabla\Psi|^{2}
+\frac{1}{4}[A_{+}(|\psi_{+}|^{2}-t^2_{+})^{2}+A_{-}(|\psi_{-}|^{2}-t^2_{-})^{2}+2B(|\psi_{+}|^{2}-t^2_{+})(|\psi_{-}|^{2}-t^2_{-})]
\end{equation}
for functions $\Psi=(\psi_+,\psi_-)\in H^1(\Omega;\mathbb C^2)$.  As is the case for any Ginzburg-Landau model, this energy is divergent for $\Omega=\RR^2$ for nontrivial entire solutions.  Nevertheless, the problem is variational in nature, and the functional $E(\Psi;\Omega)$ will be a useful tool for studying solutions locally.  

By hypothesis (H), the potential term in the energy
$$  F(\Psi) = A_{+}(|\psi_{+}|^{2}-t^2_{+})^{2}+A_{-}(|\psi_{-}|^{2}-t^2_{-})^{2}+2B(|\psi_{+}|^{2}-t^2_{+})(|\psi_{-}|^{2}-t^2_{-})  $$ 
is strictly positive definite, and attains its minimum (of zero) when $|\psi_\pm|=t_\pm$.  Following previous work on the classical Ginzburg-Landau equations (see \cite{BMR}), we seek solutions for which the potential energy is integrable,
$$  \int_{\RR^2} F(\Psi)\, dx <\infty.  $$
By the positive definiteness of $F$, this suggests that the solutions we seek should attain the asymptotic values 
$$  |\psi_\pm(x)|\to t_\pm \quad\text{as\ $|x|\to \infty$.}  $$
That is, for $|x|$ large, the solutions may be written in polar form, 
$\psi_\pm(x)=\rho_\pm(x) e^{i\phi_\pm(x)}$ with real-valued $\rho_\pm(x),\phi_\pm(x)$, and $\rho_\pm(x)\simeq t_\pm$.  The phases $\phi_\pm(x)$ may have nontrivial winding number around any large circle $C_R$ enclosing the origin, and thus these solutions carry two integer {\em degrees}  $n_\pm=\deg({\psi_\pm\over |\psi_\pm|}, C_R)\in\mathbb Z$.  
As there are no smooth $\Sigma$-valued functions with nontrivial degrees in a simply connected domain, solutions $\Psi(x)$ with nontrivial winding must vanish in one or more of its components $\psi_\pm$ to avoid singularities.  These zeros are the {\em vortices} of the solution.  

In this paper we consider special solutions of the equations \eqref{eqns} with vortices, which are equivariant with assigned degree pair $[n_+,n_-]\in\mathbb Z^2$,
$$
\pP(x)=\fP(r)e^{i\nP\tta},\ \ \pM(x)=\fM(r)e^{i\nM\tta}\ ,
$$
in polar coordinate $(r, \tta)$. By taking complex conjugates if necessary, we may assume that ${n_\pm}\ge0$. By the equivariant ansatz, the system \eqref{eqns} reduces to a system of ODEs:
\be\label{radeq}
\left.
\begin{gathered}
 \dis -f''_{+}-\frac{1}{r}f'_{+}+\frac{n^2_{+}}{r^2}f_{+}+\left[ A_{+}(f^2_{+}-t^2_{+})+B(f^2_{-}-t^2_{-})\right]f_{+}=0, \ \ {\text{for}}\ r\in (0,\ \infty), \\
 \dis -f''_{-}-\frac{1}{r}f'_{-}+\frac{n^2_{-}}{r^2}f_{-}+\left[ A_{-}(f^2_{-}-t^2_{-})+B(f^2_{+}-t^2_{+})\right]f_{-}=0, \ \ {\text{for}}\ r\in (0,\ \infty), \\
f_{\pm}(r)\ge 0 \ \text{for\ \ all}\ \ r\in [0,\ \infty), \\
 \dis f_{\pm}(r)\to t_{\pm}\ \ \text{as}\ \ r\to\infty, \\
f_{\pm}(0)=0 \ \ \text{if}\ \ n_{\pm}\neq0; \ \ \ \ f'_{\pm}(0)=0 \ \ \text{if}\ \ n_{\pm}=0.
\end{gathered}
\right\}
\ee

The first result is the existence of unique equivariant solutions, for any given degree pair $[n_+,n_-]\in\mathbb Z^2$:

\begin{prop}\label{lemunique}
Let $n_\pm\in\mathbf{ Z}$ be given and $A_{+}A_{-}-B^{2}>0$. Then there exists a unique solution $[\fP(r),\ \fM(r)]$ to \eqref{radeq} for $r\in [0,\ \infty)$.  Moreover,
\begin{gather}
f_{\pm}\in C^{\infty}\left((0,\ \infty)\right),\\
f_{\pm}(r)>0\ \text{ for all}\ r>0,\label{fpstive}\\ 
f_{\pm}(r)\sim r^{|n_\pm|}\ \text{for}\ r\sim0, \label{fnear0} \\
\int_0^\infty (f'_\pm)^2\, r\, dr <\infty, \label{gradint}\\
\int_0^\infty \left[ A_+ (f_+^2-t_+^2)^2 + 2B(f_+^2-t_+^2)(f_-^2-t_-^2) + A_- (f_-^2 - t_-^2)^2\right] r\, dr = n_+^2 t_+^2 + n_-^2 t_-^2.
\label{BMR1}
\end{gather} 
In particular, $\Psi(x)=[\fP(r)e^{i\nP\tta},\ \fM(r)e^{i\nM\tta}]$ is an entire solution of \eqref{eqns} in $\mathbb{R}^2$.
\end{prop}

We note that conclusion \eqref{BMR1} gives a quantization result (see \cite{BMR}) for the solutions in terms of the degree pair at infinity,
\begin{multline}\label{intgcond}
\dis\int_{\mathbf{ R}^2}\left[ A_{+}(|\psi_{+}|^{2}-t^2_{+})^{2}+A_{-}(|\psi_{-}|^{2}-t^2_{-})^{2}+2B(|\psi_{+}|^{2}-t^2_{+})(|\psi_{-}|^{2}-t^2_{-})\right]dx  \\
=2\pi \left(n_+^2 t_+^2 + n_-^2t_-^2\right).
\end{multline}

To determine the shape of the vortex profiles, we first consider the asymptotic form of the solutions for $r\to\infty$.  We prove that
\begin{equation}\label{as_intro}
f_{\pm}(r)=\dis t_{\pm}+\frac{a_{\pm}}{r^2}+\frac{b_{\pm}}{r^4}+O(r^{-6}), \quad
f'_{\pm}(r)=-\frac{2a_{\pm}}{r^3} + O(r^{-5}),
\quad \text{as}\ \ r\to\infty,
\end{equation}
with 
$$
a_{\pm}=\dis \frac{1}{2}\frac{Bn_\mp^2-A_{\mp}n_\pm^2}{(A_{+}A_{-}-B^{2})t_{\pm}},
$$
and (rather complicated) constants $b_\pm$ given in \eqref{bpm}. 
A formal asymptotic expansion of this form (in fact, an expansion to arbitrary order in $1/r$) may be obtained by simply substituting an ansatz into the system of equations and matching terms.  Our results, presented in Theorem~\ref{asymptotics}, provide rigorous confirmation of this expansion by means of a sub- and super-solution construction.   Moreover, we prove that the expansion is \underbar{uniform} in the coefficients $A_\pm, B, t_\pm$ lying in a compact set. The proof is completed using a new and original comparison principle (Lemma~\ref{generalcomplem}) for elliptic systems, which generalizes the one in \cite{AB}. While there is no comparison or maximum principle which applies to a generic elliptic system, in this case the equations have a special structure (similar to competitive or cooperative systems) which fortunately admits a comparison result.

From the asymptotics \eqref{as_intro}, we see that the shape of the solutions depends strongly on the coefficients, in particular the sign of the interaction coefficient $B$.  For $B<0$, both components approach their limiting value $t_\pm$ from below, as is familiar from the classical GL vortices.  However, for $B>0$, this may no longer be the case, and for certain choices of $n_\pm$ and $B$ one of the components will approach its limiting value from above.  Such behavior was already noted in \cite{ABM1} in the case $n_-=0$, for a ``balanced'' system, $A_+=A_-$, $t_+=t_-$.  
In case $n_\pm\neq 0$, our result implies that there are parameter regimes in which vortex profiles will be non-monotone.

\medskip

Finally, we consider the question of monotonicity of the solution profiles.  For the standard GL vortices, the vortex profile is known to be strictly monotone increasing in $r$.  As suggested by the asymptotic expansion above, the validity of this property is strongly dependent on the value of $B$.  We prove the following:

\begin{theorem}\label{monotone}
Let $A_+, A_->0$ be fixed, and $B$ such that $B^2<A_+A_-$.
Assume $\Psi(x;B)=[f_{+}(r;B)e^{in_{+}\tta},f_{-}(r;B)e^{in_{-}\tta}]$ is the equivariant solution for those parameters $A_\pm,B$.
\begin{enumerate}[{\bf (i)}]
\item If $B<0$, then $f'_{\pm}(r;B)\ge 0$ for all $r>0$ and for any degree $[n_{+},n_{-}]$.

\item If $B>0$, $n_{+}\ge1$ and $n_{-}=0$, then $f'_{+}(r;B)\ge0$ and $f'_{-}(r;B)\le0$ for all $r>0$.

\item  For any pair $[n_+,n_-]$ with $n_+\neq 0\neq n_-$, there exists $B_0>0$ such that $f'_\pm(r;B)\ge 0$ for all $r>0$ and all $B$ with $0\le B\le B_0$.
\end{enumerate}
\end{theorem}

As it is expected that only degrees $n_\pm=0,\pm 1$ give rise to locally minimizing solutions (see \cite{Sh}, \cite{ABM2}), let us summarize what we know about these specific symmetric vortex solutions.  The situation for the $[n_+,n_-]=[1,0]$ vortex is very clear:  when $B\le 0$ both components are monotone increasing, but when $B>0$, $f_-$ {\it decreases} monotonically, while $f_+$ increases.  This is the same result as was found in \cite{ABM1} in the special case $A_+=A_-$, $t_+=t_-$.

For the $[n_+,n_-]=[1,1]$ vortex the picture is more complicated.
For the interaction coefficient $B< 0$, both components $f_\pm(r)$ are monotone increasing to their limiting values $t_\pm$ as $r\to\infty$, and so they always resemble classical GL vortices.  When $B=0$, the equations decouple, and each component is a rescaled GL vortex profile, and hence monotone increasing as well.  For $B>0$, there are two possibilities.  For small values of $B>0$, both components continue to be monotone increasing in $r$.  But (for instance) when $A_+<B<A_-$, the component $f_+(r)$ has $f_+(0)=0$ but it approaches its limiting value {\it from above}, and hence cannot be either increasing or decreasing in $r$.  This is a phenomenon which is not familiar in the GL theory, although non-monotone vortex profiles have been found in the context of confined BEC (see \cite{KoP}).

Parts (i) and (ii) of Theorem~\ref{monotone} are proven using the second variation of energy around an equivariant solution. Part (iii) is obtained via a compactness argument based on the uniform asymptotic estimate from Theorem~\ref{asymptotics}.

\bigskip

The questions addressed in this paper are motivated by analogous results  for the classical (scalar complex) Ginzburg-Landau equations,
\begin{equation}\label{GL}  -\Delta u + (|u|^2-1)u=0.  
\end{equation}
  The existence, uniqueness and strict monotonicity of equivariant solutions to \eqref{GL} were proven in \cite{HH}, and the asymptotic expansion of the radial solutions as $r\to\infty$ was derived in \cite{CEQ}.  Results for a system of two Ginzburg-Landau equations \eqref{eqns} have been recently been obtained in \cite{AB}, \cite{ABM1}, \cite{ABM2} in the special case of ``balanced'' coefficients, $A_+=A_-$ and $t_+=t_-=1/\sqrt{2}$.  For the balanced case, existence, uniqueness and monotonicity of radial solutions are proven in \cite{ABM1}, and questions of minimality, stability and bifurcation are treated in \cite{ABM2}.  While several of the methods employed in these papers may be generalized to the general case (with hypotheses (H)), the balanced coefficient case provides several simplifications which are not valid in the general case.  For instance, in the balanced case with $n_+=n_-$, the unique equivariant solution to \eqref{eqns} takes the simple form $\Psi=\left({u\over \sqrt{2}},{u\over \sqrt{2}}\right)$ where $u$ solves the classical (complex scalar) Ginzburg-Landau equation \eqref{GL}, and thus the full treatment of the system of equations \eqref{eqns} may be reduced to dealing with a single equation.  As discussed above, the general case provides more variety of possible behaviors for symmetric vortex solutions, and requires enhanced methods which apply to systems of equations.  

A forthcoming paper will consider the stability and minimality of the symmetric vortex solutions.

\subsection*{Two-component Bose-Einstein condensates}

In a model of a two-compenent BEC from Eto {\it et al.} \cite{Eto}, we consider a pair of complex wave functions $\Phi=(\varphi_1,\varphi_2)\in\CC^2$, defined in the sample domain $\Omega\subset\RR^2$.  The energy of the configuration is defined as
$$  E(\varphi_1, \varphi_2) = 
  \int_\Omega\left[ {\hbar^2\over 2m_1}|\nabla\varphi_1|^2
     + {\hbar^2\over 2m_2}|\nabla\varphi_2|^2
      + \frac12\left( g_1|\varphi_1|^4 + g_2|\varphi_2|^4
            + 2g_{12}|\varphi_1|^2|\varphi_2|^2\right)\right] dx,
$$
where $m_1,m_2>0$ are the masses, and the coupling constants satisfy the positivity condition $g_1g_2-g_{12}^2>0$.  The Gross-Pitaevskii equations govern the dynamics of the condensate,
\begin{equation}\label{GP} \left.
\begin{aligned}
i\hbar\partial_t\varphi_1 &= -{\hbar^2\over 2m_1}\Delta\varphi_1
       + g_1|\varphi_1|^2\varphi_1 + g_{12}|\varphi_2|^2\varphi_1, \\
i\hbar\partial_t\varphi_2 &= -{\hbar^2\over 2m_2}\Delta\varphi_2
       + g_{12}|\varphi_1|^2\varphi_2 + g_{2}|\varphi_2|^2\varphi_2.
\end{aligned}
\right\}
\end{equation}
A stationary equation of the desired form is obtained by considering standing wave solutions, $\varphi_i(x,t)= e^{-i\mu_i t/\hbar}u_i(x)$, $i=1,2$, where $\mu_i$ represent the chemical potentials:
\begin{align*}
 -{\hbar^2\over 2m_1}\Delta\varphi_1
       + g_1|\varphi_1|^2\varphi_1 + g_{12}|\varphi_2|^2\varphi_1
       &= \mu_1\varphi_1, \\
 -{\hbar^2\over 2m_2}\Delta\varphi_2
       + g_{12}|\varphi_1|^2\varphi_2 + g_{2}|\varphi_2|^2\varphi_2
        &=  \mu_2\varphi_2.
\end{align*}
In the variational formulation of the stationary problem, the chemical potentials represent Lagrange multipliers, which arise because of the constraints on the masses of the two condensate species, 
$$  \int_\Omega |\varphi_i|^2 dx = \int_\Omega |u_i|^2 dx = N_i, \quad i=1,2.  $$
By a rescaling of the dependent variables, $\psi_+=\sqrt[4]{{m_2\over m_1}}\,u_1$, $\psi_-=\sqrt[4]{{m_1\over m_2}}\,u_2$, we may eliminate the masses $m_i$ from the equations, and we obtain the system 
\begin{align*}
-\eps^2\Delta\psi_{+}+[A_{+}(|\psi_{+}|^{2}-t^2_{+})+B(|\psi_{-}|^{2}-t^2_{-})]\psi_{+}&=0,\\ 
-\eps^2\Delta\psi_{-}+[A_{-}(|\psi_{-}|^{2}-t^2_{-})+B(|\psi_{+}|^{2}-t^2_{+})]\psi_{-}&=0,
\end{align*}
with $\eps^2={\hbar^2\over \sqrt{m_1m_2}}$, $A_+={m_1\over m_2}g_1$, $A_-={m_2\over m_1}g_2$, $B=g_{12}$, and 
$$ t_+^2= {\mu_1 g_2-\mu_2 g_{12}\over g_1g_2-g_{12}^2}     \sqrt{{m_2\over m_1}}, \qquad
     t_-^2= 
     {\mu_2 g_1-\mu_1 g_{12}\over g_1g_2-g_{12}^2}\sqrt{{m_1\over m_2}} .
$$
As is typical for GL equations, by blowing up around the core of a vortex at scale $\eps$, we obtain an entire solution of \eqref{eqns} in all of $\RR^2$. 

In a more physically appropriate model for a two-component BEC \cite{KTU}, the Laplacian in the Gross-Pitaevskii system \eqref{GP} should be replaced by the Hamiltonians 
$$ H_i:= -{\hbar^2\over 2m_i} \Delta + V_i^{\rm trap} - \omega L_z, 
\qquad i=1,2,  $$
with harmonic  trapping potentials $V_i^{\rm trap}=c_i^2 |x|^2/2$, $c_i$ constant $i=1,2$; angular momentum operator $L_z$; and (constant) angular speed $\Omega$.  While this is an essential step, both in modeling the confinement of the condensate and in describing the onset of vortices in the sample, these terms will not affect the general form \eqref{eqns} of the blow-up equations which describe the vortex profiles at length scale $\eps$ in the condensate.  Indeed, the momentum operator plays much the same role as the magnetic vector potential in the GL model of superconductivity, and for rotations which are of moderate strength in $\eps$, $\omega\ll \eps^{-1}$, the analysis of Proposition~3.12 of \cite{SS} may be used to derive \eqref{eqns} in limit $\eps\to 0$ after rescaling.

\section{Existence and uniqueness}

We begin with an {\it a priori} bound on solutions of the system \eqref{eqns} in bounded domains.  Let $\lambda_{s}>0$ denote the smallest eigenvalue of the matrix $\left[\begin{array}{cc}A_{+} & B \\B & A_{-}\end{array}\right]$.

\begin{prop}\label{uniformbound}
Let $\Omega\subset\RR^2$ be a bounded smooth domain, and 
$\Psi=(\psi_+,\psi_-)$ be a solution of \eqref{eqns} in $\Omega$, satisfying
\begin{equation}\label{BC}
  |\Psi(x)|^2 = |\psi_+(x)|^2 + |\psi_-(x)|^2 \le t_+^2 + t_-^2 \quad
\text{for all $x\in\partial\Omega$}.
\end{equation}
Then $|\Psi(x)|^{2}\le\min\{{2M\over \lambda_{s}}, t^2_{+}+t^2_{-} \}$ for all $x\in\Omega$, where $M=\max\{A_{+}t^2_{+}+Bt^2_{-}, A_{-}t^2_{-}+Bt^2_{+}\}$.
 \end{prop}
 
We note that the bound on $\|\Psi\|_\infty$ depends only on the coefficients in the energy, and in particular it is independent of the domain $\Omega$.
Furthermore, the quantity $M>0$ whenever $B^2<A_+A_-$.

\begin{proof}
Let $V(x)=|\Psi|^{2}$.
We then calculate:
\begin{align}\nnn
\dis {1\over2}\Delta V&=(\psi_{+},\nabla\psi_{+})+(\psi_{-},\nabla\psi_{-})+|\nabla\Psi|^2\nnn\\
                                      &\ge (\psi_{+},\nabla\psi_{+})+(\psi_{-},\nabla\psi_{-})\nnn\\
                                      &\dis =[A_{+}(|\psi_{+}|^{2}-t^2_{+})+B(|\psi_{-}|^{2}-t^2_{-})]|\psi_+|^2 +[A_{-}(|\psi_{-}|^{2}-t^2_{-})+B(|\psi_{+}|^{2}-t^2_{+})]|\psi_-|^2\nnn\\
                                      &\dis=\left[A_{+}|\psi_{+}|^{4}+2B|\psi_{+}|^{2}|\psi_{-}|^{2}+A_{-}|\psi_{-}|^{4}-(A_{+}t^2_{+}+Bt^2_{-})|\psi_{+}|^{2}-(A_{-}t^2_{-}+Bt^2_{+})|\psi_{-}|^{2} \right]\nnn\\
                                      &\dis\ge  \left[\lambda_{s}(|\psi_{+}|^{4}+|\psi_{-}|^{4})-(A_{+}t^2_{+}+Bt^2_{-})|\psi_{+}|^{2}-(A_{-}t^2_{-}+Bt^2_{+})|\psi_{-}|^{2} \right]\nnn\\
                                      &\dis \ge \left({\lambda_{s}\over 2}|\Psi|^{4}-M|\Psi|^2\right)\nnn\\
                                      &\dis=\frac{\lambda_{s}}{2}V\left(V-\frac{2M}{\lambda_{s}}\right).\label{Vineq}
\end{align}

Now let $v=V-\frac{2M}{\lambda_{s}}$, which is then a subsolution of the following boundary value problem:
$$
\left\{
\begin{gathered}\nnn
-\Delta v+a(x)v\le 0 \ \ \ \ \ \ \mbox{in}\ \ \ \Omega,\nnn\\
v\Big|_{\partial\Omega}\le t^2_{+}+t^2_{-}-\frac{2M}{\lambda_{s}},
\end{gathered}
\right.
$$
with $a(x)=\lambda_{s}V>0$.

If $t^2_{+}+t^2_{-}-\frac{2M}{\lambda_{s}}\le0$, by maximum principle $v\le0$ on $\Omega$, i.e. $V\le\frac{2M}{\lambda_{s}}$. On the other hand, if $t^2_{+}+t^2_{-}-\frac{2M}{\lambda_{s}}>0$, by maximum principle again $\dis\max_{\overline{\Omega}}v$ can be attained on $\partial\Omega$, i.e. $\dis\max_{\overline{\Omega}}v\le t^2_{+}+t^2_{-}-\frac{2M}{\lambda_{s}}$ and $v\le \dis t^2_{+}+t^2_{-}-\frac{2M}{\lambda_{s}}$, i.e. $V\le t^2_{+}+t^2_{-}$. We thus conclude that $V\le\dis\min\left\{{2M \over \lambda_{s}}, t^2_{+}+t^2_{-}\right\}$, i.e. $|\Psi|^{2}\le\dis\min\left\{{2M \over \lambda_{s}}, t^2_{+}+t^2_{-}\right\}$.
\end{proof}

We may now prove the existence and uniqueness theorem.

\begin{proof}[Proof of Proposition~\ref{lemunique}]
Without loss of generality, we assume $n_\pm\ge 0$, as negative degrees may be obtained by complex conjugation of $\psi_\pm$.
To obtain the existence we consider the problem defined in the ball $B_{R},\ R>0$,
\be\label{radeqR}
\left\{
\begin{array}{l}
 \dis -f''_{\pm}-\frac{1}{r}f'_{\pm}+\frac{n^2_{\pm}}{r^2}f_{\pm}+\left[ A_{\pm}(f^2_{\pm}-t^2_{\pm})+B(f^2_{\mp}-t^2_{\mp})\right]f_{\pm}=0, \ \ {\text{for}}\ 0<r<R, \\
 \dis f_{\pm}(R)= t_{\pm}, \\
f_{\pm}(0)=0 \ \ \text{if}\ \ n_{\pm}\neq0; \ \ \ \ f'_{\pm}(0)=0 \ \ \text{if}\ \ n_{\pm}=0.
\end{array}
\right.
\ee

The existence of such a solution follows easily by  minimization of the energy 
\be\label{radenergy}
\left.\begin{array}{l}
E^{R}_{n_{+},\ n_{-}}(\fP,\ \fM) \\
\ \ \ \ \ \dis =\frac12\int^R_0\left\{\sum_{i=\pm}\left[(f'_{i})^{2}+\frac{n^2_{\pm}}{r^2}f^2_{i}+\frac12[A_{i}(f^2_{i}-t^2_{i})^{2}+2B(f^2_{+}-t^2_{+})(f^2_{-}-t^2_{-})]\right]\right\}rdr\ ,
\end{array}
\right.
\ee
over Sobolev functions satisfying the appropriate boundary conditions at $r=0$ and $r=R$. Denote $[f_{R,+}(r),\ f_{R,-}(r)]$ as any solution of \eqref{radeqR}.  Since $E^{R}_{n_{+},\ n_{-}}(|f_+|,|f_-|)=E^{R}_{n_{+},\ n_{-}}(\fP,\ \fM)$, and there exists a minimizer which is nonnegative, $f_{\pm}\ge0$ in $[0, R]$.

By Proposition~\ref{uniformbound}, the solution pair is bounded, $f^2_{R,+}(r)+f^2_{R,-}(r)\le\Lambda^2:=\min\{{2M\over \lambda_s},t_+^2+t_-^2\}$, for any solution to \eqref{radeq}, uniformly in $R$.  By $L^p$ estimates (applied to each equation separately), we may conclude that $f_{R,\pm}$ is uniformly bounded in $W^{2,p}_{loc}$ for any $p$, and by elliptic regularity the family is bounded in $\mathcal{C}^k_{loc}$ for any $k$.  Thus, there exists a subsequence $R_{n}\to\infty$ for which the solution $[f_{R,+}(r),\ f_{R,-}(r)]\to[f_{\infty,+}(r),\ f_{\infty,-}(r)]$  in $\mathcal{C}^{k}_{\text{loc}}([0,\infty))$, and the limit functions $[f_{\infty,+}(r),\ f_{\infty,-}(r)]$ give (weak) solutions to the ODE on $(0,\infty)$ with the same boundary condition at $r=0$ and satisfying the uniform bound $f^2_{\infty,+}(r)+f^2_{\infty,-}(r)\le\Lambda$ for all $r\in [0,\infty)$.  By uniform convergence, we also have $f_{\infty,\pm}(r)\ge 0$ for all $r\in [0,\infty)$.

To determine the behavior near $r=0$, we write the ODE in the form:
$$
\dis -f''_{\pm}-\frac{f'_{\pm}}{r}+\frac{n^2_{\pm}}{r^2}f_{\pm}=g_\pm(r),
$$ 
with $g_\pm$ smooth, and  $g_\pm(0)=0$ for $n_\pm\neq 0$.
By the Frobenius theory (see Appendix~3  of \cite{BC} or Lemma~5.9 of \cite{ABM2}), we deduce that the behavior $f_{\infty,\pm}\sim r^{n_\pm}$ near $r=0$, which is \eqref{fnear0}.  As a consequence (see Theorem~3.1 of \cite{BC}), we may conclude that $\psi_{\pm}(x)=f_{\infty,\pm}(r)e^{in_{\pm}\tta}$ is regular at $r=0$ and solves \eqref{eqns} in $\mathbb{ R}^2$.

Next we must verify that $f_{\infty,\pm}$ attains the desired asymptotic conditions as $r\to\infty$.  We first claim that
\begin{equation}\label{intcond}
\int_0^\infty \left[ (f_{\infty,+}^2(r)-t_+^2)^2 + (f_{\infty,-}^2(r)-t_-^2)^2\right] r\, dr
<\infty.
\end{equation}
For this purpose, we derive a Pohozaev identity: we multiply the equation of $f_{R,\pm}$ by $r^{2}f'_{R,\pm}(r)$ and integrate by parts with respect to $r\in (0,R)$. Adding the two resulting identities gives:
\begin{multline}\label{Pohozaev}
 [Rf'_{R,+}(R)]^{2}+[Rf'_{R,-}(R)]^{2} \\
 +\int^{R}_{0}\left\{A_{+}(f^2_{R,+}-t^2_{+})^2+A_{-}(f^2_{R,-}-t^2_{-})^2
+2B(f^2_{R,+}-t^2_{+})(f^2_{R,-}-t^2_{-})\right\}
rdr=
\\  n^2_{+}t^2_{+}+n^2_{-}t^2_{-},
\end{multline}
which is the Pohozaev identity.
By the uniform convergence $f_{R,\pm}\to f_{\infty,\pm}$ on $[0,R_{0}]$ for any $R_{0}>0$ we have
\begin{align*}
&\lambda_s\int_0^{R_0} \left[ (f_{\infty,+}^2(r)-t_+^2)^2 + (f_{\infty,-}^2(r)-t_-^2)^2\right] r\, dr 
\\ &\qquad\le
\int^{R_0}_{0}\left\{A_{+}(f^2_{\infty,+}-t^2_{+})^2+A_{-}(f^2_{\infty,-}-t^2_{-})^2+2B(f^2_{\infty,+}-t^2_{+})(f^2_{\infty,-}-t^2_{-})\right\}rdr \\
&\qquad\le n^2_{+}t^2_{+}+n^2_{-}t^2_{-},
\end{align*}
for any fixed $R_0>0$. Letting $R_{0}\to\infty$ we establish the condition \eqref{intcond}. 

Next, we prove \eqref{gradint}.   In the remainder of the proof, we write $f_\pm=f_{\infty,\pm}$.  Let $\eta_1(r)$ be a smooth function with $\eta_1(r)=1$ for $0\le r\le 1$ and $\eta_1(r)=0$ when $r\ge 2$, and set $\eta_R(r):=\eta_1(r/R)$.
We then multiply the equation for $f_\pm$ by $r(t_\pm-f_\pm)\eta_R$  and integrate by parts to obtain:
\begin{align*}
&\int_0^\infty\left[
  {n_\pm^2\over r^2} f_\pm(t_\pm-f_\pm)\eta_R +
     A_\pm(f_\pm^2-t_\pm^2)f_\pm(t_\pm-f_\pm)
     + B(f_\mp^2-t_\mp^2)f_\pm(t_\pm-f_\pm)\right] r\, dr \\
     &\qquad 
       = \int_0^\infty \left(f'_\pm\right)^2\eta_R\, r\, dr
           + \frac12 \int_0^\infty \left( (t_\pm-f_\pm)^2\right)' \eta'_R\, r\, dr
           \\
     &\qquad
       =  \int_0^\infty \left(f'_\pm\right)^2\eta_R\, r\, dr
           - \frac12 \int_0^\infty (t_\pm-f_\pm)^2\, \Delta_r\eta_R\, r\, dr,
\end{align*}
where $\Delta_{r}f:=\frac{1}{r}(rf'(r))'$ is  the Laplacian for radial functions.  
We now estimate each term, using \eqref{intcond}.  We note that, as $f_\pm\ge 0$, we have 
$(t_\pm^2 - f_\pm^2)^2 = (t_\pm+f_\pm)^2(t_\pm-f_\pm)^2 \ge
t_\pm^2(t_\pm-f_\pm)^2$,
and hence 
$\| t_\pm-f_\pm \|_{L^2}\le {1\over t_\pm}\| t_\pm^2-f_\pm^2\|_{L^2}<\infty$ by \eqref{intcond}.  Thus, we obtain the bounds (uniform in $R$),
\begin{gather*}
\int_0^\infty (t_\pm - f_\pm)^2 |\Delta_r\eta_R|\, r\, dr \le {C\over R^2}
  \int_0^\infty (t_\pm - f_\pm)^2\, r\, dr \to 0, \\
\left|\int_0^\infty {n_\pm^2\over r^2} f_\pm(t_\pm-f_\pm)\eta_R\, r\, dr
\right| \le \Lambda \left\| {n_\pm^2\over r^2}\right\|_{L^2}
    \| t_\pm-f_\pm\|_{L^2} \le C_1, \\
\left|  \int_0^\infty A_\pm (f_\pm^2-t_\pm^2)f_\pm(t_\pm-f_\pm)\, r\, dr
\right| \le A_\pm \Lambda \| f_\pm^2 - t_\pm^2\|_{L^2} 
   \| f_\pm-t_\pm\|_{L^2} \le C_2,
\end{gather*}
and similarly for the $B$ term in the integral.  We thus conclude that 
$$ \int_0^\infty \left(f'_\pm\right)^2\eta_R\, r\, dr\le C_3 $$
is bounded independently of $R$, and the finiteness of the integral \eqref{gradint} follows from Fatou's lemma.

The limit statement $f_{\pm}(r)\to t_\pm$ as $r\to\infty$ now follows from Proposition~\ref{uniformbound} and the bounds \eqref{intcond} and \eqref{gradint}.  Indeed, for any $0<r_1<r_2$,
\begin{align*}
(f_{\pm}^2(r_2)-t_\pm^2)^2 - 
(f_{\pm}^2(r_1)-t_\pm^2)^2 & = 
\int_{r_1}^{r_2} 4(f_{\pm}^2(r)-t_\pm^2){f_{\pm}(r)\over r} f'_{\pm}\, r\, dr \\
&\le 2{\Lambda\over r_1}\int_{r_1}^{r_2} \left[
    \left(f'_{\pm}(r)\right)^2 + (f_{\pm}^2-t_\pm^2)^2\right]
       \, r\, dr,
\end{align*}
and hence $(f_{\pm}^2(r)-t_\pm^2)^2$ is Cauchy as $r\to\infty$.

For the existence part of the theorem, it remains to verify \eqref{BMR1}, we note that by \eqref{gradint} there must exist a sequence $R_n\to\infty$ such that $R_n f'_{\pm}(R_n)\to 0$.  Applying \eqref{Pohozaev} with $R=R_n$, and passing to the $n\to\infty$ limit gives \eqref{BMR1}.This completes the existence part of Proposition~\ref{lemunique}.

Uniqueness may be proven in (at least) two different ways.  The first employs the method of Brezis and Oswald \cite{BO}, which we now sketch. Let $[n_{+},n_{-}]\in \mathbb{ Z}^2$ to be given, and suppose $[f_{+},\ f_{-}]$ and $[g_{+},\ g_{-}]$ are two solutions of \eqref{radeq}. Dividing the equations for $f_\pm$ by $f_\pm$ (and similarly for $g_\pm$), multiplying by $f^2_{+}-g^2_{+}$ and integrating by parts, we obtain:

\begin{align*}
&-\int^{\infty}_{0}\left[A_{+}(f^2_{+}-g^2_{+})^{2}+B(f^2_{+}-g^2_{+})(f^2_{-}-g^2_{-})\right]rdr \\
& \qquad\qquad\qquad =
\int^{\infty}_{0}\left[ -\frac{\Delta_{r}(f_{+})}{f_{+}}+\frac{\Delta_{r}(g_{+})}{g_{+}}\right](f^2_{+}-g^2_{+})rdr \\
&\qquad\qquad\qquad
=\dis \int^{\infty}_{0}\left[(f'_{+})^{2}-2\frac{f_{+}}{g_{+}}f'_{+}g'_{+}+\frac{f^2_{+}}{g^2_{+}}(g'_{+})^{2}\right]rdr\\
&\quad\qquad\qquad\qquad
 + \int^{\infty}_{0}\left[(g'_{+})^{2}-2\frac{g_{+}}{f_{+}}f'_{+}g'_{+}+\frac{g^2_{+}}{f^2_{+}}(f'_{+})^{2}\right]rdr\\
&
\qquad\qquad\qquad=\dis \int^{\infty}_{0}\left(\left|f'_{+}-\frac{f_{+}}{g_{+}}g'_{+}\right|^2 +\left|g'_{+}-\frac{g_{+}}{f_{+}}f'_{+}\right|^2\right)
rdr,
\end{align*}
and similarly,
\begin{multline*}
\dis \int^{\infty}_{0}\left(
\left|f'_{-}-\frac{f_{-}}{g_{-}}g'_{-}\right|^2+\left|g'_{-}-\frac{g_{-}}{f_{-}}f'_{-}\right|^2\right) rdr
\\
 =\dis -\int^{\infty}_{0}\left[A_{-}(f^2_{-}-g^2_{-})^{2}+B(f^2_{+}-g^2_{+})(f^2_{-}-g^2_{-})\right]rdr.
\end{multline*}
Adding the above identities together we obtain that
\begin{equation}
\left.
\begin{array}{rl}
0&\le\dis\int^{\infty}_{0}\left(\left|f'_{+}-\frac{f_{+}}{g_{+}}g'_{+}\right|^2 +\left|g'_{+}-\frac{g_{+}}{f_{+}}f'_{+}\right|^2+\left|f'_{-}-\frac{f_{-}}{g_{-}}g'_{-}\right|^2 +\left|g'_{-}-\frac{g_{-}}{f_{-}}f'_{-}\right|^2 \right)rdr\\
&=\dis -\int^{\infty}_{0}\left\{A_{+}(f^2_{+}-g^2_{+})^{2}+A_{-}(f^2_{-}-g^2_{-})^{2}+2B(f^2_{+}-g^2_{+})(f^2_{-}-g^2_{-})\right\}rdr\le 0,
\end{array}
\right.
\end{equation}
as the quadratic form on the right-hand side is positive definite. 
Therefore, 
 $f^{2}_{\pm}-g^{2}_{\pm}\equiv0$, i.e. $f_{\pm}\equiv g_{\pm}$ for all $r\in (0,\infty)$, and this completes the first proof of uniqueness.
 
 For a second proof, we note that the system of equations for $f_\pm$ may be written in the form
$$  \left. \begin{gathered}
-\Delta f_+ + \partial_u G(r,f_+^2, f_-^2)f_+ =0, \\
-\Delta f_- + \partial_v G(r,f_+^2, f_-^2)f_- =0, \\
f_\pm(r)\ge 0,
\end{gathered}
\right\}
$$
 with
$$ G(r,u,v) = {n_+^2\over r^2}u + {n_-^2\over r^2} v +
\frac12\left( A_+(u-t_+^2)^2 + 2B(u-t_+^2)(v-t_-^2) + A_-(v-t_-^2)^2\right).
$$
As $G$ is convex in $(u,v)$, we may apply Theorem~4.1 of \cite{ABM1} (together with the limiting argument given in that paper, to apply the theorem to the semi-infinite interval) to conclude uniqueness.  
\end{proof}

\section{Asymptotics}

We derive the asymptotic behavior of solutions $f_\pm(r)$ as $r\to\infty$ by means of the sub- and super-solutions method.  The basic tool we require is the following useful and original comparison lemma:

\begin{lemma}\label{generalcomplem}
Let $\mathbb{A, B, C, D}$ be bounded functions on $[R,\infty)$.

$\mathrm{(A)}$ Assume $\mathbb{A}$, $\mathbb{D}>0$, $\mathbb{B}$, $\mathbb{C}\le0$ and $4\mathbb{AD}-(\mathbb{B+C})^2>0$ in $[R,\infty)$. Then, if $u, v$ are radial solutions of the following problem
$$
\left\{
\begin{gathered}
\dis -\Delta_r u+ {m^2\over r^2}u+\mathbb{A}u+\mathbb{B} v\le 0\ \ \text{in}\ \ \ [R,\infty),\\
\dis -\Delta_r v+{n^2\over r^2}v+\mathbb{C}u+\mathbb{D} v\le 0\ \ \text{in}\ \ \ [R,\infty),
\end{gathered}
\right.
$$
with $u(R)\le0, v(R)\le0$, $u, v$ are bounded in $[R, \infty)$ with $\int^{\infty}_{R}(u')^2 rdr<\infty$ and $\int^{\infty}_{R}(v')^2 rdr<\infty$, we have that $u\le0$ and $v\le0$ in $[R, \infty)$.

$\mathrm{(B)}$ Assume $\mathbb{A}$, $\mathbb{D}>0$, $\mathbb{B}$, $\mathbb{C}\ge0$ and $4\mathbb{AD}-(\mathbb{B+C})^2>0$ in $[R,\infty)$. Then, if $u, v$ are radial solutions of the following problem
$$
\left\{
\begin{gathered}
\dis -\Delta_r u+{m^2\over r^2}u+\mathbb{A}u+\mathbb{B} v\le 0\ \ \text{in}\ \ \ [R,\infty),\\
\dis -\Delta_r v+{n^2\over r^2}v+\mathbb{C}u+\mathbb{D} v\ge 0\ \ \text{in}\ \ \ [R,\infty),
\end{gathered}
\right.
$$
with $u(R)\le0\le v(R)$, $u, v$ are bounded in $[R, \infty)$ with $\int^{\infty}_{R}(u')^2 rdr<\infty$ and $\int^{\infty}_{R}(v')^2 rdr<\infty$, we have that $u\le0\le v$ in $[R, \infty)$.
\end{lemma}

\begin{proof}
To verify (A), multiply the respective equations by $u^{+}=\text{max} (u, 0)$ and $v^{+}=\text{max} (v, 0)$ and integrate by parts in $[R, T]$ for $\forall T\ge R>0$. Since $u$ and $v$ are radial solutions, we have
$$
\dis -\int^T_{R}\Delta u\cdot u^{+} dr=-\int^T_{R} (ru')'u^{+} dr=-\left(u'u^{+}r\big|^T_{R}\right)+\int^T_{R}[(u^{+})']^{2}rdr ,
$$
hence, 
$$
\dis -\left(u'u^{+}r\big|^T_{R}\right)+\int^T_{R}\left\{
  [(u^{+})']^{2}+ {m^2\over r^2}(u^+)^2
  +[\mathbb{A}(u^{+})^{2}+\mathbb{B}(v^{+}+v^{-})u^{+}]\right\}
  rdr\le0 ,
$$
i.e.
\begin{align}
\dis &\int^T_{R}\left\{
  [(u^{+})']^{2}+ {m^2\over r^2}(u^+)^2
  +[\mathbb{A}(u^{+})^{2}+\mathbb{B}(v^{+}+v^{-})u^{+}]\right\}
  rdr\nnn\\
&\dis\qquad \le u'(T)u^{+}(T)T-u'(R)u^{+}(R)R\nnn\\
& \dis\qquad \le u'(T)u^{+}(T)T\label{upineq}\ ,
\end{align}
since $u'(R)u^{+}(R)R=0$ by $u(R)\le 0, u^{+}(R)=0$. 

We now  claim that there exists $T_{n}\to\infty$ such that $u'(T_{n})T_{n}\to 0$.  Indeed, if not, there exist $c_0,T_0>0$ with $u'(r)r\ge c_{0}$ for all $r\ge T_{0}$. It follows that $u'(r)\ge\frac{c_{0}}{r}$, then we obtain that $[u'(r)]^{2}r\ge\frac{c^2_{0}}{r}\notin L^{1}$, which is a contradiction.

Combining the boundness of $u$ in $[R,\infty)$ and the result of the above claim, letting $T\to\infty$ on the both sides of \eqref{upineq}, we get 
\begin{equation}\label{upinfty}
\dis \int^{\infty}_{R}\left\{ [(u^{+})']^{2}rdr +{m^2\over r^2} (u^+)^2+[\mathbb{A}(u^{+})^{2}+\mathbb{B}u^{+}v^{+}+\mathbb{B}u^{+}v^{-}]\right\}rdr\le 0\ .
\end{equation}
Similarly, we have the inequality for $v^+$:
\begin{equation}\label{vpinfty}
\dis \int^{\infty}_{R}\left\{
 [(v^{+})']^{2}rdr +{n^2\over r^2} (v^+)^2 +
 [\mathbb{D}(v^{+})^{2}+\mathbb{C}u^{+}v^{+}+\mathbb{C}u^{-}v^{+}]\right\}rdr\le 0\ .
\end{equation}
Since $\mathbb{B}u^{+}v^{-}>0$, $\mathbb{C}u^{-}v^{+}>0$ with $u^{+}=\text{max} (u, 0)>0$,  $v^{+}=\text{max} (v, 0)>0$, $v^{-}=\text{min} (v, 0)<0$, we get that
$$
\dis\int^{\infty}_{R}\left\{
 [(u^{+})']^{2}+{m^2\over r^2}(u^+)^2+ [\mathbb{A}(u^{+})^{2}+\mathbb{B}u^{+}v^{+}]\right\} rdr
 \le 0,
$$ 
and similarly,
$$
\dis \int^{\infty}_{R}\left\{
 [(v^{+})']^{2}+ {n^2\over r^2}(v^+)^2 +
 [\mathbb{D}(v^{+})^{2}+\mathbb{C}u^{+}v^{+}]rdr\right\}rdr \le 0.
$$
Therefore, we deduce that
\begin{equation}\label{uvpineq}
\dis \int^{\infty}_{R} \left\{[(u^{+})']^{2}+[(v^{+})']^{2}
+ {m^2 (u^+)^2 + n^2 (v^+)^2\over r^2}
+\mathbb{A}(u^{+})^{2}+(\mathbb{B+C})u^{+}v^{+}+\mathbb{D}(v^{+})^{2}\right\}rdr\le 0\ .
\end{equation}
Note that the matrix associated to the quadratic form is positive definite in $[R,\infty)$ by hypothesis, so there exists a function $\lambda^{+}>0$ with
$$
\mathbb{A}(u^{+})^{2}+(\mathbb{B+C})u^{+}v^{+}+\mathbb{D}(v^{+})^{2}\ge \lambda^{+}\left[(u^{+})^{2}+(v^{+})^{2}\right]>0\ .
$$
In consequence,
$$
0<\dis \int^{\infty}_{R}  \left\{[(u^{+})']^{2}+[(v^{+})']^{2}
+ {m^2 (u^+)^2 + n^2 (v^+)^2\over r^2}
+\lambda^{+}\left[(u^{+})^{2}+(v^{+})^{2}\right]\right\}rdr\le0\ ,
$$
which implies that $u^{+}\equiv0\equiv v^{+}$ in $[R,\infty)$. Therefore $u(r)\le 0$ and $v(r)\le 0$ in $[R,\infty)$, and case (A) is verified.

\medskip

To prove (B) we multiply the first inequality by $u^+$, but multiply the second inequality by $v^{-}=\text{min}(v, 0)\le 0$. Then, integrate by parts, we get that
$$
\dis \int^T_{R} \left\{[(u^{+})']^{2}+\frac{m^2}{r^2}(u^+ )^{2}\right\}rdr+\int^T_{R} [\mathbb{A}(u^{+})^{2}+\mathbb{B}u^{+}v^{+}+\mathbb{B}u^{+}v^{-}]rdr\le u'(T)u^{+}(T)T
$$
and
$$
\dis \int^T_{R} \left\{[(v^{-})']^{2}+\frac{n^2}{r^2}(v^- )^{2}\right\}rdr+\int^T_{R} [\mathbb{C}u^{+}v^{-}+\mathbb{C}u^{-}v^{-}+\mathbb{D}(v^{-})^{2}]rdr\le v'(T)v^{-}(T)T
$$
by $v(R)\ge 0$ and $v^{-}(R)=0$.  By the same argument as in the above claim, there exists $T_{n}\to\infty$ such that $v'(T_{n})T_{n}\to 0$.  Hence, passing to the limit $T_n\to\infty$, we have
\begin{equation}\label{vminfty}
\dis \int^\infty_{R} \left\{[(v^{-})']^{2}+\frac{n^2}{r^2}(v^- )^{2}\right\}rdr+\int^\infty_{R} [\mathbb{C}u^{+}v^{-}+\mathbb{C}u^{-}v^{-}+\mathbb{D}(v^{-})^{2}]rdr\le 0.
\end{equation}
Since $\mathbb{B}u^{+}v^{+}\ge 0$, $\mathbb{C}u^{-}v^{-}\ge 0$ with $u^{+}=\text{max}(u, 0)>0$, $u^{-}=\text{min}(u, 0)<0$, $v^{+}=\text{max}(v, 0)>0$, $v^{-}=\text{min}(v, 0)<0$, we can drop the two terms $\mathbb{B}u^{+}v^{+}\ge 0$ and $\mathbb{C}u^{-}v^{-}\ge 0$ without affecting the inequalities \eqref{upinfty} and \eqref{vminfty}. We add \eqref{upinfty} and \eqref{vminfty} together, and note that the matrix associated to the quadratic form is positive definite in $[R, \infty)$ by hypothesis, so there exists a function $\lambda^{+}>0$ with
$$
\mathbb{A}(u^{+})^{2}+(\mathbb{B+C})u^{+}v^{-}+\mathbb{D}(v^{-})^{2}\ge \lambda^{+}[(u^{+})^{2}+(v^{-})^{2}]>0\ .
$$
Therefore, we obtain that
\begin{align*}
&\dis \int^{\infty}_{R} \left\{ [(u^{+})']^{2}+[(v^{-})']^{2}+
{m^2 (u^+)^2 + n^2 (v^-)^2\over r^2} +
\lambda^{+}[(u^{+})^{2}+(v^{-})^{2}]\right\}rdr\\
&\qquad \dis \le \int^{\infty}_{R} \left\{ [(u^{+})']^{2}+[(v^{-})']^{2}
+{m^2 (u^+)^2 + n^2 (v^-)^2\over r^2} 
+\mathbb{A}(u^{+})^{2}+(\mathbb{B+C})u^{+}v^{-}+\mathbb{D}(v^{-})^{2}\right\}rdr
\\  &\qquad\qquad\le 0\ .
\end{align*}
Consequently,
 $u^{+}\equiv 0\equiv v^{-}$ in $[R,\infty)$, and hence $u(r)\le 0$ and $v(r)\ge 0$ in $[R,\infty)$.
\end{proof}

After this preliminary, we introduce the asymptotics of the radial solution at $\infty$.

\begin{theorem}\label{asymptotics}
Let $[f_{+}, f_{-}]$ be the solution of \eqref{radeq} with degree pair $[n_{+},n_{-}]$ at $\infty$, then we have 
$$
f_{\pm}=\dis t_{\pm}+\frac{a_{\pm}}{r^2}+\frac{b_{\pm}}{r^4}+O(r^{-6})\ \ \ \ \text{as}\ \ r\to\infty,
$$
with 
\begin{equation}\label{apm}
a_{\pm}=\dis \frac{1}{2}\frac{Bn_\mp^2-A_{\mp}n_\pm^2}{(A_{+}A_{-}-B^{2})t_{\pm}},
\end{equation}
and 
\begin{equation}\label{bpm}
b_{\pm}=\dis -\frac{A^2_{\mp}(8n_\pm^2 +n_\pm^4)t^2_{\mp}
-BA_{\mp}(2n_\pm^2+8)n_\mp^2t^2_{\mp}
-8BA_{\pm}n_\mp^2t^2_{\pm}
+B^{2}(8t_\pm^2n_\pm^2+n_\mp^4t_\mp^2)t^2_{\pm}}{8(A_{+}A_{-}-B^2)^2 t^3_{\pm}t^2_{\mp}}.
\end{equation}
More specifically, let $A_{\pm}\ge0$ and $B_{0}>0$ so that $A_{+}A_{-}-B^{2}_0 >0$. Then, there exist positive constants $C_{1}, C_{2}, R$ such that
\begin{equation}\label{fpmasybound}
\dis\left|f_{\pm}(r)-\left(t_{\pm}+\frac{a_{\pm}}{r^2}+\frac{b_{\pm}}{r^4}\right)\right|\le\frac{C_{1}}{r^6} ,
\end{equation}
\begin{equation}\label{f'pmasybound}
\dis\left|f'_{\pm}(r)+\frac{2a_{\pm}}{r^3}\right|\le\frac{C_{2}}{r^5} ,
\end{equation}
hold for all $r\ge R$ and all $B$, $|B|\le B_0$.
\end{theorem}

\begin{proof}
We treat the cases $B>0$ and $B<0$ separately. We start with $B>0$.

{\bf Step 1:} Construction of a sub-supsolution pair. Let
\begin{align}
\dis \bar{w}_{+}=t_{+}+a_{+}\frac{R^2}{r^2}+b_{+}\frac{R^4}{r^4}+\bar{c}_{+}\frac{R^6}{r^6},\ \label{wbarp}\\
\dis \underline{w}_{-}=t_{-}+a_{-}\frac{R^2}{r^2}+b_{-}\frac{R^4}{r^4}+\underline{c}_{-}\frac{R^6}{r^6},\ \label{wunderm}
\end{align}
where $a_{\pm}, b_{\pm}, \bar{c}_{+}, \underline{c}_{-}$ and $R$ are to be chosen so that
\begin{align}
\text{LHS($\bar{w}_{+}$)}=\dis -\bar{w}''_{+}-\frac{\bar{w}'_{+}}{r}+\frac{n_+^2}{r^2}\bar{w}_{+}+[A_{+}(\bar{w}^2_{+}-t^2_{+})+B(\underline{w}^2_{-}-t^2_{-})]\bar{w}_{+}\ge 0, \label{LHSwbarp}\\
\text{LHS($\underline{w}_{-}$)}=\dis -\underline{w}''_{-}-\frac{\underline{w}'_{-}}{r}+\frac{n_-^2}{r^2}\underline{w}_{-}+[A_{-}(\underline{w}^2_{-}-t^2_{-})+B(\bar{w}^2_{+}-t^2_{+})]\underline{w}_{-}\le 0, \label{LHSwunderm}
\end{align}
for all $r\ge R$, and $f_{+}(R)\le \bar{w}_{+}(R)$, $f_{-}(R)\ge \underline{w}_{-}(R)$.

Expanding \eqref{LHSwbarp} and \eqref{LHSwunderm}, we obtain terms which are polynomials in even powers of $r^{-1}$, 
\begin{equation}\label{Lwbarpsum}
\text{LHS($\bar{w}_{+}$)}=\dis \sum^{9}_{k=1} M^{+}_{2k}\left( {R\over r}\right)^{2k}, 
\end{equation}
\begin{equation}\label{Lwundermsum}
\text{LHS($\underline{w}_{-}$)}=\dis \sum^{9}_{k=1} M^{-}_{2k}\left( {R\over r}\right)^{2k},
\end{equation}
where $M^{\pm}_{2k}=M^{\pm}_{2k}(A_{\pm}, B, R, a_{\pm}, b_{\pm}, \bar{c}_{+}, \underline{c}_{-})$ is a polynomial in its arguments. 
The coefficients are quite intricate, but may be explicitly calculated with the aid of Maple software.

 First, we choose $a_\pm$ in order to force the lowest order coefficients
 $M_2^\pm$ to vanish:  that is, 
\begin{equation*}
M^{\pm}_{2}=[n_\pm^2+2(A_{\pm}t_{\pm}a_{\pm}+Bt_{\mp}a_{\mp})R^2 ]t_{\pm}=0
\end{equation*}
is achieved by choosing $a_\pm$ as in \eqref{apm}.
Similarly, we fix the values of the coefficients $b_\pm$ in order that the next term
\begin{equation*}
M^{\pm}_{4}=[(n_\pm^2-4)a_{\pm}R^{2}+2R^{4}(A_{\pm}b_{\pm}t^2_{\pm}+Bb_{\mp}t_{\pm}t_{\mp}+Ba_{\pm}a_{\mp}t_{\mp})+R^{4}t_{\pm}(3A_{\pm}a^2_{\pm}+Ba^2_{\mp})]t_{\pm}=0,
\end{equation*}
which determines the values for $b_\pm$ given in \eqref{bpm}.

Again using Maple, the values of $a_{\pm}, b_{\pm}$ may then be substituted into the expansions of \eqref{LHSwbarp} and \eqref{LHSwunderm}, and the expressions for $M^\pm_{2k}$ may be viewed as functions of $R$. The exact form of the coefficients $M^\pm_{2k}$ is very complex, but they are all rational expressions of the coefficients $A_\pm,B,t_\pm,n_\pm$, and as we will choose $R$ large, we are only interested in the leading order of each. We obtain:
$$
\begin{gathered}
M^{+}_{6}=2(B\underline{c}_{-}t_{-}+A_{+}\bar{c}_{+}t_{+})t_{+}+O(R^{-6}), \\
M^{-}_{6}=2(B\bar{c}_{+}t_{+}+A_{-}\underline{c}_{-}t_{-})t_{-}+O(R^{-6}), \\
M^{\pm}_{8}=O(R^{-2}),\ \ \ \ \ \ \ M^{\pm}_{10}=O(R^{-4}),\\
M^{+}_{12}=(A_{+}\bar{c}^2_{+}+B\underline{c}^2_{-})t_{+}+2(B\underline{c}_{-}t_{-}+A_{+}\bar{c}_{+}t_{+})\bar{c}_{+}+O(R^{-6}),\\
M^{-}_{12}=(A_{-}\underline{c}^2_{-}+B\bar{c}^2_{+})t_{-}+2(B\bar{c}_{+}t_{+}+A_{-}\underline{c}_{-}t_{-})\underline{c}_{-}+O(R^{-6}),\\
M^{\pm}_{14}=O(R^{-2}),\ \ \ \ \ \ \ M^{\pm}_{16}=O(R^{-4}),\\
M^{+}_{18}=(A_{+}\bar{c}^2_{+}+B\underline{c}^2_{-})\bar{c}_{+},\\
M^{-}_{18}=(A_{-}\underline{c}^2_{-}+B\bar{c}^2_{+})\underline{c}_{-},
\end{gathered}
$$
here $O(R^{-n})$ denotes terms which are small as $R\to\infty$ uniformly for $|B|\le B_0$.

As $M^{\pm}_{6}$ is the leading order term, we choose $\bar{c}_{+}, \underline{c}_{-}$ in order that it gives the correct signs, and so that it dominates the other terms for $r\ge R$.  To this end, we let $[\tilde{c}_{+}, \tilde{c}_{-}]$ be the unique solutions to the following system:
$$
\left\{
\begin{gathered}
A_{+}t_{+}\tilde{c}_{+}+Bt_{-}\tilde{c}_{-}=1\ ,\\
Bt_{+}\tilde{c}_{+}+A_{-}t_{-}\tilde{c}_{-}=-1\ .
\end{gathered}
\right.
$$
We note that, 
\begin{equation}
\tilde{c}_{+}=\dis\frac{A_{-}+B}{(A_{+}A_{-}-B^2)t_+}>0\ ,\ \ \ \ \ \ \tilde{c}_{-}=\dis-\frac{A_{+}+B}{(A_{+}A_{-}-B^2)t_-}<0\ .
\end{equation}
Let $\bar{c}_{+}=\delta\tilde{c}_{+},\ \underline{c}_{-}=\delta\tilde{c}_{-}$ with $0<\delta<1$ to be chosen later, hence $\bar{c}_{+}>0$, $\underline{c}_{-}<0$, and 
\begin{equation}\label{M6pm}
M^+_{6}=2\delta t_{+}>0,\ \ \ \ \ \ M^-_{6}=-2\delta t_{-}<0.
\end{equation}
By choosing $R$ sufficiently large, we obtain that 
\begin{equation}\label{4bound}
|M^{\pm}_{8}|, |M^{\pm}_{10}|, |M^{\pm}_{14}|, |M^{\pm}_{16}|\le \dis {1\over 20} |M^{\pm}_{6}|,
\end{equation}
and so we have that
\begin{multline}\label{4sumbound}
\dis \left|M^{\pm}_{8}\left({R\over r}\right)^{8}+M^{\pm}_{10}\left({R\over r}\right)^{10}+M^{\pm}_{14}\left({R\over r}\right)^{14}+M^{\pm}_{16}\left({R\over r}\right)^{16}\right|\\
\le {1\over 5} |M^{\pm}_{6}|\left({R\over r}\right)^{6}, \ \ \ \text{for}\ \ \forall r\ge R.
\end{multline}

Also,
\begin{align}\nnn
\dis |M^+_{12}|&=|\delta^{2}(A_{+}\tilde{c}^2_{+}+B\tilde{c}^2_{-})t_{+}+2\delta^{2}\tilde{c}_{+}|\nnn\\
                      &=\delta^{2}|(A_{+}\tilde{c}^2_{+}+B\tilde{c}^2_{-})t_{+}+2\tilde{c}_{+}|\nnn\\
                      &\le c_{1}(A_{\pm},B, t_{\pm})\delta^2\nnn\\
                      &\le {2\over 5} \delta t_+= \dis {1\over 5}|M^+_6|, \label{M12pp} 
\end{align}
for $R$ sufficiently large, provided we choose $\delta\le {2t_{+} \over 5c_1}$.

Meanwhile, 
\begin{align}\nnn
\dis |M^-_{12}|&=|\delta^{2}(A_{+}\tilde{c}^2_{+}+B\tilde{c}^2_{-})t_{+}+2\delta^{2}\tilde{c}_{+}|\nnn\\
                  &=\delta^{2}|(A_{-}\tilde{c}^2_{-}+B\tilde{c}^2_{+})t_{-}+2\tilde{c}_{-}|\nnn\\
                 &\le c_{2}(A_{\pm},B, t_{\pm})\delta^2\nnn\\
                 &\le {2\over 5} \delta t_-=\dis {1\over 5}|M^-_6| , \label{M12mp} 
\end{align}
for $R$ sufficiently large, provided we choose $\delta\le {2t_{-} \over 5c_2}$. Therefore, if we choose $\delta\le\min\{{2t_{+} \over 5c_1}, {2t_{-} \over 5c_2}\}$, it yields that $|M^\pm_{12}|\le {1\over 5}|M^\pm_6 |$. So we can deduce that 
\begin{equation}\label{M12}
\dis \left| M^\pm_{12} \left({R\over r}\right)^{12} \right|\le {1\over 5}\left|M^\pm_{6}  \left({R\over r}\right)^6 \right|.
\end{equation}
Finally, we note that  $M^\pm_{18}$ has the same sign as $M^\pm_6$, and so it contributes with the desired sign.

By \eqref{Lwbarpsum}-\eqref{Lwundermsum} and \eqref{M6pm}, \eqref{4bound}, \eqref{4sumbound}, \eqref{M12}, we deduce that
\begin{align}\nnn
\dis \sum^{9}_{k=1} M^{+}_{2k}\left( {R\over r}\right)^{2k}&\ge -{1\over 5}M^+_{6}\left({R\over r}\right)^{6}-{1\over 5}M^+_{6}\left({R\over r}\right)^{6}+M^+_{6}\left({R\over r}\right)^{6}\nnn\\
&=\dis {3\over 5}M^+_{6}\left({R\over r}\right)^{6}={6\over 5}\delta t_{+}\left({R\over r}\right)^{6}\nnn\\
&>\dis {3\over 5}\delta t_{+}\left({R\over r}\right)^{6}>0 ,
\end{align}
\begin{align}\nnn
\dis \sum^{9}_{k=1} M^{-}_{2k}\left( {R\over r}\right)^{2k}&\le {1\over 5}|M^-_{6}|\left({R\over r}\right)^{6}+{1\over 5}|M^-_{6}|\left({R\over r}\right)^{6}+M^-_{6}\left({R\over r}\right)^{6}\nnn\\
&\dis =-{1\over 5}M^-_{6}\left({R\over r}\right)^{6}-{1\over 5}M^-_{6}\left({R\over r}\right)^{6}+{1\over 5}M^-_{6}\left({R\over r}\right)^{6}\nnn\\
&\dis = {3\over 5}M^-_{6}\left({R\over r}\right)^{6}=-{6\over 5}\delta t_{-}\left({R\over r}\right)^{6}\nnn\\
&\dis\le -{3\over 5}\delta t_{-}\left({R\over r}\right)^{6}<0 ,
\end{align}
for all $r\ge R$. Thus, $\bar{w}_+$, $\underline{w}_-$ satisfy \eqref{LHSwbarp} and \eqref{LHSwunderm}, as desired.

Now we consider $\bar{w}_{+}$ and $\underline{w}_{-}$ at $r=R$. Combine \eqref{wbarp}, \eqref{wunderm} and \eqref{apm} we have $\bar{w}_{+}=t_{+}+\bar{c}_{+}-O(R^{-2})$, $\underline{w}_{-}=t_{-}+\underline{c}_{-}-O(R^{-2})$. Since $f_{\pm}(r)\to t_{\pm}$ as $r\to\infty$ and $\bar{c}_{+}>0, \underline{c}_{-}<0$, we deduce that $\bar{w}_{+}(R)>f_{+}(R)$ and $\underline{w}_{-}(R)<f_{-}(R)$ for $R\ge R_0$ sufficiently large. This completes Step 1.

{\bf Step 2:} We will apply Lemma~\ref{generalcomplem} to show that $\bar{w}_{+}(r)$ is a super-solution to $f_{+}$-equation of \eqref{radeq} and $\underline{w}_{-}(r)$ is a sub-solution to $f_{-}$-equation of \eqref{radeq}. Let $u=f_{+}-\bar{w}_{+}$, $v=f_{-}-\underline{w}_{-}$, together with equations \eqref{radeq}, \eqref{LHSwbarp} and \eqref{LHSwunderm}, and denote
$$  L_\pm u:= -\Delta_r u + {n_\pm^2\over r^2}u.  $$
We then have that:
\begin{align*}
L_+ u   &=A_{+}(t^2_{+}-f^2_{+})f_{+}+B(t^2_{-}-f^2_{-})f_{+}+A_{+}(\bar{w}^2_{+}-t^2_{+})\bar{w}_{+}+B(\underline{w}^2_{-}-t^2_{-})\bar{w}_{+}\\
                      &=A_{+}(\bar{w}^3_{+}-f^3_{+})-A_{+}t^2_{+}(\bar{w}_{+}-f_{+})+Bt^2_{-}(f_{+}-\bar{w}_{+})+B(\underline{w}^2_{-}\bar{w}_{+}-f^2_{-}f_{+})\\
                      &=[A_{+}t^2_{+}+Bt^2_{-}-A_{+}(\bar{w}^2_{+}+\bar{w}_{+}f_{+}+f^2_{+})]u+B(\underline{w}^2_{-}\bar{w}_{+}-f^2_{-}f_{+})\\
                      &=[A_{+}t^2_{+}+B(t^2_{-}-f^2_{-})-A_{+}(\bar{w}^2_{+}+\bar{w}_{+}f_{+}+f^2_{+})]u-B(\underline{w}_{-}+f_{-})\bar{w}_{+}v\ ,
\end{align*}
and
\begin{align*}
L_- v     &=A_{-}(t^2_{-}-f^2_{-})f_{-}+B(t^2_{+}-f^2_{+})f_{-}+A_{-}(\underline{w}^2_{-}-t^2_{-})\underline{w}_{-}+B(\bar{w}^2_{+}-t^2_{+})w_{-}\\
                      &=A_{-}(\underline{w}^3_{-}-f^3_{-})-A_{-}t^2_{-}(\underline{w}_{-}-f_{-})+Bt^2_{+}(f_{-}-\underline{w}_{-})+B(\bar{w}^2_{+}\underline{w}_{-}-f^2_{+}f_{-})\\
                      &=[A_{-}t^2_{-}+B(t^2_{+}-f^2_{+})-A_{-}(\underline{w}^2_{-}+\underline{w}_{-}f_{-}+f^2_{-})]v-B(\bar{w}_{+}+f_{+})\underline{w}_{-}u\ ,
\end{align*}
and thus \eqref{LHSwbarp} and \eqref{LHSwunderm} imply that
\begin{equation}
L_+u +[A_{+}(\bar{w}^2_{+}+\bar{w}_{+}f_{+}+f^2_{+})-A_{+}t^2_{+}+B(f^2_{-}-t^2_{-})]u+B(\underline{w}_{-}+f_{-})\bar{w}_{+}v\le 0\ ,
\end{equation}
\begin{equation}
L_- v+B(\bar{w}_{+}+f_{+})\underline{w}_{-}u+ [A_{-}(\underline{w}^2_{-}+\underline{w}_{-}f_{-}+f^2_{-})-A_{-}t^2_{-}+B(f^2_{+}-t^2_{+})]v\ge0\ .     
\end{equation}          
As in Lemma~\ref{generalcomplem}, we have a system on $[R,\infty)$ of
the form:       
$$
\left\{
\begin{gathered}
L_+ u+\mathbb{A}(r)u+\mathbb{B}(r)v\le0, \ \ \ \ u(R)\le0,\\
L_- v+\mathbb{C}(r)u+\mathbb{D}(r)v\ge0, \ \ \ \ v(R)\ge0,
\end{gathered}
\right.
$$
with 
$$
\begin{gathered}
\mathbb{A}(r)=A_{+}(\bar{w}^2_{+}+\bar{w}_{+}f_{+}+f^2_{+})-A_{+}t^2_{+}+B(f^2_{-}-t^2_{-}),\\
\mathbb{B}(r)=B(\underline{w}_{-}+f_{-})\bar{w}_{+},\ \ \ \ \mathbb{C}(r)=B(\bar{w}_{+}+f_{+})\underline{w}_{-},\\
\mathbb{D}(r)=A_{-}(\underline{w}^2_{-}+\underline{w}_{-}f_{-}+f^2_{-})-A_{-}t^2_{-}+B(f^2_{+}-t^2_{+}).
\end{gathered}
$$          
Now, we check the hypothesis of Lemma~\ref{generalcomplem} with the uniform convergence of $f_{\pm}(r)$ as $r\to\infty$:
$$
\begin{gathered}
\mathbb{A}(r)\to 2A_{+}t^2_{+},\ \ \ \  \ \ \ \ \ \ \mathbb{B}(r)\to 2Bt_{+}t_{-}, \\
\mathbb{C}(r)\to 2Bt_{+}t_{-}, \ \ \ \ \ \ \ \mathbb{D}(r)\to 2A_{-}t^2_{-},
\end{gathered}
$$         
and 
$$
4\mathbb{AD}-(\mathbb{B+C})^{2}\to 4\cdot 2A_{+}t^2_{+}\cdot2A_{-}t^2_{-}-(4Bt_{+}t_{-})^{2}\\
=16t^2_{+}t^2_{-}(A_{+}A_{-}-B^{2})>0,
$$
for $R$ sufficiently large.   
All conditions of Lemma~\ref{generalcomplem} are satisfied, it yields that $u(r)\le0, v(r)\ge0$ in $[R,\infty)$, i.e. $f_{+}(r)\le \bar{w}_{+}(r), f_{-}(r)\ge \underline{w}_{-}(r)$ in $[R,\infty)$.

{\bf Step 4:} We repeat the above procedures with the roles $\pm$ of $\bar{w}_{+}$ and $\underline{w}_-$ under the condition $B>0$. We may obtain $\underline{c}_{+}, \bar{c}_{-}, R$ so that $f_{+}(r)\ge \underline{w}_{+}(r), f_{-}(r)\le \bar{w}_{-}(r)$ in $[R,\infty)$. Combine the results from above steps, we get that
\begin{equation}\label{f+ss}
\underline{w}_{+}(r)\le f_{+}(r)\le \bar{w}_{+}(r),\ \ \ \underline{w}_{-}(r)\le f_{-}(r)\le \bar{w}_{-}(r), \ \ \ \text{for}\ \ \forall r\ge R,
\end{equation}
in the case $B>0$.

{\bf Step 5:}  Now we consider the case $B<0$. This is similar to the previous cases, but we must set up part (A) of the Comparison Lemma~\ref{generalcomplem}. Firstly we construct a super-solution pair, with $\bar{w}_{+}$ as in \eqref{wbarp}, $\bar{w}_{-}$ as following, 
\begin{equation}\label{wbarm}
\bar{w}_{-}=\dis t_{-}+a_{-}{R^{2}\over r^2}+b_{-}{R^{4}\over r^4}+\bar{c}_{-}{R^{6}\over r^6},
\end{equation}
with $a_{\pm}, b_{\pm}, \bar{c}_{\pm}$ and $R$ are to be chosen so that 
\begin{align}
\text{LHS($\bar{w}_{+}$)}=\dis -\bar{w}''_{+}-\frac{\bar{w}'_{+}}{r}+\frac{n_+^2}{r^2}\bar{w}_{+}+[A_{+}(\bar{w}^2_{+}-t^2_{+})+B(\bar{w}^2_{-}-t^2_{-})]\bar{w}_{+}\ge 0, \label{LHSwdbarpp}\\
\text{LHS($\bar{w}_{-}$)}=\dis -\bar{w}''_{-}-\frac{\bar{w}'_{-}}{r}+\frac{n_-^2}{r^2}\bar{w}_{-}+[A_{-}(\bar{w}^2_{-}-t^2_{-})+B(\bar{w}^2_{+}-t^2_{+})]\bar{w}_{-}\ge 0, \label{LHSwdbarmp}
\end{align}
for all $r\ge R$, and $f_{+}(R)\le\bar{w}_{+}(R), f_{-}(R)\le\bar{w}_{-}(R)$.

Using Maple software, we expand \eqref{LHSwdbarpp} and \eqref{LHSwdbarmp}, which is a polynomial in even power of $r^{-1}$, as in \eqref{Lwbarpsum} and 
\begin{align}
\text{LHS($\bar{w}_{-}$)}=\dis \sum^9_{k=1}M^{-}_{2k}\left({R\over r}\right)^{2k}, \label{Lwbarmsum}
\end{align}
respectively, where $M^{\pm}_{2k}=M^{\pm}_{2k}(A_{\pm}, B, R, a_{\pm}, b_{\pm}, \bar{c}_{\pm})$ is a polynomial in its arguments. Similarly, we choose $a_{\pm}, b_{\pm}$ as in \eqref{apm} and \eqref{bpm} separately so that $M^{\pm}_{2}$  may be zero, arriving at  the same formulas in Step 1. Then, repeat the same procedure as in Step 1, we get the exact forms for $M^{\pm}_{6}, M^{\pm}_{12}$ and $M^{\pm}_{18}$ as the same as in Step 1 but with $\underline{c}_{-}$ replaced by $\bar{c}_{-}$, and $M^{\pm}_{8}, M^{\pm}_{10}$ and $M^{\pm}_{14}$ as the same as in Step 1.

As $M^{\pm}_6$ is still the leading order term, we choose $\bar{c}_{\pm}$ in order that it gives the correct sign, and so that it dominates the other terms for all $r\ge R$.  To do this,  we take $[\hat{c}_{+}, \hat{c}_{-}]$ to be the unique solutions of the linear system:
$$
\left\{
\begin{gathered}
A_{+}t_{+}\hat{c}_{+}+Bt_{-}\hat{c}_{-}=1,\\
Bt_{+}\hat{c}_{+}+A_{-}t_{-}\hat{c}_{-}=1,
\end{gathered}
\right.
$$
that is:
$$
\hat{c}_{+}=\dis\frac{A_{-}-B}{(A_{+}A_{-}-B^2)t_{+}}>0,\ \ \ \ \ \ \hat{c}_{-}=\dis\frac{A_{+}-B}{(A_{+}A_{-}-B^2)t_{-}}>0.
$$
Let $\bar{c}_{\pm}=\delta\hat{c}_{\pm}$ with $0<\delta<1$ to be chosen later, hence $\bar{c}_{\pm}>0$ 
and 
\begin{equation}\label{M6pp}
M^{\pm}_{6}=2\delta t_{\pm}>0.
\end{equation}
By choosing $R$ large enough, we can also get \eqref{4bound}, \eqref{4sumbound} and
\begin{align}\nnn
\dis |M^+_{12}|&=|\delta^{2}(A_{+}\hat{c}^2_{+}+B\hat{c}^2_{-})t_{+}+2\delta^{2}\hat{c}_{+}|\nnn\\
                      &=\delta^{2}|(A_{+}\hat{c}^2_{+}+B\hat{c}^2_{-})t_{+}+2\hat{c}_{+}|\nnn\\
                      &\le c_{5}(A_{\pm},B, t_{\pm})\delta^2\nnn\\
                      &\le\dis {1\over 5}M^+_6={2\over 5} \delta t_+ , \label{m12pp} 
\end{align}
which implies that $\delta\le {2t_{+} \over 5c_5}$.

Meanwhile, 
\begin{align}\nnn
\dis |M^-_{12}|&=|\delta^{2}(A_{-}\hat{c}^2_{-}+B\hat{c}^2_{+})t_{-}+2\delta^{2}\hat{c}_{-}|\nnn\\
                  &=\delta^{2}|(A_{-}\hat{c}^2_{-}+B\hat{c}^2_{+})t_{-}+2\hat{c}_{-}|\nnn\\
                 &\le c_{6}(A_{\pm},B, t_{\pm})\delta^2\nnn\\
                 &\le\dis {1\over 5}M^-_6={2\over 5} \delta t_-, \label{m12mp} 
\end{align}
which implies that $\delta\le {2t_{-} \over 5c_6}$. Therefore, if we choose $\delta\le\min\{ {2t_{+} \over 5c_5}, {2t_{-} \over 5c_6}\}$, it follows that $|M^{\pm}_{12}|\le {1\over 5}|M^{\pm}_{6}|$. And for $M^{\pm}_{18}$, we have
\begin{align}\nnn
\dis |M^{\pm}_{18}|=\delta^{3}(A_{\pm}\hat{c}^2_{\pm}+B\hat{c}^2_{\mp})\hat{c}_{\pm}=O(\delta^3),
\end{align}
uniformly in $A_\pm,B$, and thus it can still be controlled by $M^{\pm}_{6}$.  In particular, taking $\delta$ smaller if necessary, we have
$|M_{18}^\pm|\le \frac15|M_6^\pm|$.

By \eqref{Lwbarpsum}, \eqref{Lwbarmsum} and \eqref{M6pp}, \eqref{4bound}-\eqref{4sumbound} and \eqref{M12}, we deduce that
\begin{align}\nnn
\dis \sum^{9}_{k=1} M^{\pm}_{2k}\left( {R\over r}\right)^{2k}&\ge-{1\over 5}M^{\pm}_{6}\left({R\over r}\right)^{6}-{1\over 5}M^{\pm}_{6}\left({R\over r}\right)^{6}+M^{\pm}_{6}\left({R\over r}\right)^{6}\nnn\\
&=\dis {3\over 5}M^{\pm}_{6}\left({R\over r}\right)^{6}={6\over 5}\delta t_{\pm}\left({R\over r}\right)^{6}>0,
\end{align}
for all $r\ge R$.

Now we consider $\bar{w}_{\pm}$ at $r=R$. Combine \eqref{wbarp}, \eqref{wbarm} and \eqref{apm}, we have $\bar{w}_{\pm}=t_{\pm}+\bar{c}_{\pm}-O(R^{-2})$. Since $f_{\pm}\to t_{\pm}$ as $r\to\infty$ and $\bar{c}_{\pm}>0$, we obtain that $\bar{w}_{\pm}(R)>f_{\pm}(R)$ for $R\ge R_{0}$ sufficiently large. This finishes Step 5.

{\bf Step 6:} We will apply Lemma~\ref{generalcomplem} to show that $\bar{w}_{\pm}(r)$ is supsolution to \eqref{eqns}. Let $u=f_{+}-\bar{w}_{+}$, $v=f_{-}-\bar{w}_{-}$, we have that
\begin{align*}
L_+ u
                      &=[A_{+}(t^2_{+}-f^2_{+})+B(t^2_{-}-f^2_{-})]f_{+}+[A_{+}(\bar{w}^2_{+}-t^2_{+})+B(\bar{w}^2_{-}-t^2_{-})]\bar{w}_{+}\\
                      &=-A_{+}(f^3_{+}-\bar{w}^3_{+})+A_{+}t^2_{+}(f_{+}-\bar{w}_{+})+Bt^2_{-}(f_{+}-\bar{w}_{+})+B(\bar{w}^2_{-}\bar{w}_{+}-f^2_{-}f_{+})\\
                      &=[-A_{+}(f^2_{+}+f_{+}\bar{w}_{+}+\bar{w}^2_{+})+A_{+}t^2_{+}+Bt^2_{-}-Bf^2_{-}]u-B(\bar{w}_{-}+f_{-})\bar{w}_{+}v\\
                      &=[A_{+}t^2_{+}+B(t^2_{-}-f^2_{-})-A_{+}(f^2_{+}+f_{+}\bar{w}_{+}+\bar{w}^2_{+})]u-B(\bar{w}_{-}+f_{-})\bar{w}_{+}v ,
\end{align*}
and 
\begin{align*}
                   L_-v   &=[A_{-}(t^2_{-}-f^2_{-})+B(t^2_{+}-f^2_{+})]f_{-}+[A_{-}(\bar{w}^2_{-}-t^2_{-})+B(\bar{w}^2_{+}-t^2_{+})]\bar{w}_{-}\\
                      &=-A_{-}(f^3_{-}-\bar{w}^3_{-})+A_{-}t^2_{-}(f_{-}-\bar{w}_{-})+Bt^2_{+}(f_{-}-\bar{w}_{-})+B(\bar{w}^2_{+}\bar{w}_{-}-f^2_{+}f_{-})\\
                      &=[-A_{-}(f^2_{-}+f_{-}\bar{w}_{-}+\bar{w}^2_{-})+A_{-}t^2_{-}+Bt^2_{+}-Bf^2_{+}]v-B(\bar{w}_{+}+f_{+})\bar{w}_{-}u\\
                      &=-B(\bar{w}_{+}+f_{+})\bar{w}_{-}u+[-A_{-}(f^2_{-}+f_{-}\bar{w}_{-}+\bar{w}^2_{-})+A_{-}t^2_{-}+B(t^2_{+}-f^2_{+})]v ,
\end{align*}
thus by \eqref{LHSwdbarpp}-\eqref{LHSwdbarmp} we have that
\begin{equation*}
L_+ u+[A_{+}(f^2_{+}+f_{+}\bar{w}_{+}+\bar{w}^2_{+})+B(f^2_{-}-t^2_{-})-A_{+}t^2_{+}]u+B(\bar{w}_{-}+f_{-})\bar{w}_{+}v\le 0,
\end{equation*}
\begin{equation*}
L_- v+B(\bar{w}_{+}+f_{+})\bar{w}_{-}u+[A_{-}(f^2_{-}+f_{-}\bar{w}_{-}+\bar{w}^2_{-})+B(f^2_{+}-t^2_{+})-A_{-}t^2_{-}]v\le 0,
\end{equation*}
and $u(R)=f_{+}(R)-\bar{w}_{+}(R)\le 0$, $v(R)=f_{-}(R)-\bar{w}_{-}(R)\le 0$. Therefore, we have the exact forms of the inequality system as in the part (A) of Lemma~\ref{generalcomplem} with 
\begin{equation}
\mathbb{A}(r)=A_{+}(f^2_{+}+f_{+}\bar{w}_{+}+\bar{w}^2_{+})+B(f^2_{-}-t^2_{-})-A_{+}t^2_{+}, 
\end{equation}
\begin{equation}
\mathbb{B}(r)=-B(\bar{w}_{-}+f_{-})\bar{w}_{+},\ \ \ \ \ \mathbb{C}(r)=-B(\bar{w}_{+}+f_{+})\bar{w}_{-},
\end{equation}
\begin{equation}
\mathbb{D}(r)=A_{-}(f^2_{-}+f_{-}\bar{w}_{-}+\bar{w}^2_{-})+B(f^2_{+}-t^2_{+})-A_{-}t^2_{-}.
\end{equation}
 
With the uniform convergence of $f_{\pm}(r)$ as $r\to\infty$, we obtain that  
\begin{equation*}
\mathbb{A}(r)\to 2A_{+}t^2_{+},\ \ \ \ \ \  \ \ \mathbb{B}(r)\to -2Bt_{+}t_{-}, 
\end{equation*}
\begin{equation*}
\mathbb{C}(r)\to -2Bt_{+}t_{-},\ \ \  \mathbb{D}(r)\to 2A_{-}t^2_{-}
\end{equation*}
and
\begin{multline*}
4\mathbb{AD}-(\mathbb{B+C})^2 \to 4\cdot 2A_{+}t^2_{+}\cdot 2A_{-}t^2_{-}-(4Bt_{+}t_{-})^2\\
=16t^2_{+}t^2_{-}(A_{+}A_{-}-B^{2})>0\ \ \ \ \ \ \text{for}\ \ \ R\ \ \text{large\ \ enough}.
\end{multline*}
Hence we can simply apply part (A) of Lemma~\ref{generalcomplem} to get that $u(r)\le0$, $v(r)\le0$ in $[R, \infty)$, i.e. $f_{\pm}(r)\le \bar{w}_{\pm}(r)$ in $[R, \infty)$.

{\bf Step 7:} We repeat the same procedure as in Step 5 and Step 6, but with a pair of sub-solutions $\underline{w}_{\pm}$ to \eqref{eqns} instead of sup-solutions $\bar{w}_{\pm}$. Let $\underline{w}_{+}$ be defined as in  the following
\begin{equation}\label{wunderp}
\underline{w}_{+}=\dis t_{+}+a_{+}{R^{2}\over r^2}+b_{+}{R^{4}\over r^4}+\underline{c}_{+}{R^{6}\over r^6},
\end{equation}
and $\underline{w}_{-}$ defined as in \eqref{wunderm}, with $a_{\pm}, b_{\pm}, \underline{c}_{\pm}$ and $R$ are to be chosen so that 
\begin{align}
\text{LHS($\underline{w}_{\pm}$)}=\dis -\underline{w}''_{\pm}-\frac{\underline{w}'_{\pm}}{r}+\frac{n^2_\pm}{r^2}\underline{w}_{\pm}+[A_{\pm}(\underline{w}^2_{\pm}-t^2_{\pm})+B(\underline{w}^2_{\mp}-t^2_{\mp})]\underline{w}_{\pm}\le 0, \label{LHSwdundermm}
\end{align}
for all $r\ge R$, and $\underline{w}_{\pm}(R)\le f_{\pm}(R)$. And we choose $\underline{c}_{\pm}=-\delta\hat{c}_{\pm}$ with the same $\delta$ as in Step 5 so that $M^{\pm}_{6}$ is still the leading terms in \eqref{Lwbarpsum} and \eqref{Lwbarmsum} with which $\bar{w}_{\pm}$ replaced by $\underline{w}_{\pm}$. Therefore, we can simply get that
\begin{align}
\text{LHS}(\underline{w}_{\pm})&\le \dis M^{\pm}_{6}\left({R\over r}\right)^{6}+{1\over5}|M^{\pm}_{6}|\left({R\over r}\right)^{6}+{1\over5}|M^{\pm}_{6}|\left({R\over r}\right)^{6}\nnn\\
                                                        &=\dis\left( M^{\pm}_{6}-{1\over 5}M^{\pm}_{6}-{1\over 5}M^{\pm}_{6}\right)\left({R\over r}\right)^{6} \nnn\\
                                                        &= \dis {3\over5}M^{\pm}_{6}\left({R\over r}\right)^{6} = -{6\over 5}\delta t_{\pm}\left({R\over r}\right)^{6} <0.
\end{align}

Next, with $u=f_{+}-\underline{w}_{+}$, $v=f_{-}-\underline{w}_{-}$, we repeat the similar process as in Step 6 and apply part (A) of Lemma~\ref{generalcomplem} again, we can obtain that $\underline{w}_{\pm}(r)\le f_{\pm}(r)$ in $[R, \infty)$.

{\bf Step 8:} Combing all the results from above steps, we have that
\begin{equation}
\underline{w}_{\pm}(r)\le f_{\pm}(r)\le \bar{w}_{\pm}(r)\ \ \ \ \ \text{in}\ \ \ [R,\infty).
\end{equation}
for both cases of $B$.

From the sub- and super-solution argument, we have that
$$
\dis\left|f_{\pm}-\left(t_{\pm}+\frac{\tilde{a}_{\pm}}{r^2}+\frac{\tilde{b}_{\pm}}{r^4}\right)\right|\le\dis \frac{C_{1}}{r^6},
$$          
with $\tilde{a}_{\pm}=a_{\pm}R^{2}, \tilde{b}_{\pm}=b_{\pm}R^4$, $a_{\pm}$ and $b_{\pm}$ defined as same as in \eqref{apm} and \eqref{bpm} respectively, 
and $C_{1}=\min\{c_{\pm}R^6\}$ with $c_{+}=\min\{|\bar{c}_{+}|, |\underline{c}_{+}|\}$, $c_{-}=\min\{|\underline{c}_{-}|, |\bar{c}_{-}|\}$.

Next we want to show that $|f'_{\pm}(r)+\frac{2\tilde{a}_{\pm}}{r^3}|\le\frac{C_2}{r^5}$ for some constant $C_{2}$. Following the idea in \cite{CEQ}, let $W_{\pm}(r)=f_{\pm}-\left(t_{\pm}+\frac{\tilde{a}_{\pm}}{r^2}+\frac{\tilde{b}_{\pm}}{r^4}\right)=f_{\pm}-w_{\pm}+\frac{c_{\pm}}{r^6}$(for convenience, we drop bar and underline of $w_{\pm}$ and $c_{\pm}$ as in the formulas of $w_{\pm}$ shown in previous steps), hence $W_{\pm}(r)=O(r^{-6})$. Therefore, we deduce that
\begin{align}
&\dis -W''_{+}-\frac{1}{r}W'_{+}+\frac{n^2_+}{r^2}W_{+}\nnn\\
&\quad=\dis -\left(f''_{+}-w''_{+}+\frac{{c}_{+}}{r^8}\right)-\frac{1}{r}\left(f'_{+}-w'_{+}+\frac{{c}_{+}}{r^7}\right)+\frac{n^2_+}{r^2}\left(f_{+}-w_{+}+\frac{c_{+}}{r^6}\right)\nnn\\
&\quad\dis=-f''_{+}-\frac{1}{r}f'_{+}+\frac{n^2_+}{r^2}f_{+}-\left(-w''_{+}-\frac{1}{r}w'_{+}+\frac{n^2_+}{r^2}w_{+}\right)+O(r^{-8})\nnn\\
&\quad=\dis[A_{+}(t^2_{+}-f^2_{+})+B(t^2_{-}-f^2_{-})]f_{+}+[A_{+}(w^2_{+}-t^2_{+})+B(w^2_{-}-t^2_{-})]w_{+}+O(r^{-8})\nnn\\
&\quad=\dis-\mathbb{A}W_{+}-\mathbb{B}W_{-}+\mathbb{A}\frac{c_{+}}{r^6}+\mathbb{B}\frac{c_{-}}{r^6}+O(r^{-8}),\nnn
\end{align}
i.e.
$$
\dis -W''_{+}-\frac{1}{r}W'_{+}+\frac{n^2_+}{r^2}W_{+}+\mathbb{A}W_{+}+\mathbb{B}W_{-}= O(r^{-6})\ \ \ \text{for}\ \ \ r\ge R.
$$

Similarly, we have
$$
\dis -W''_{-}-\frac{1}{r}W'_{-}+\frac{n^2_-}{r^2}W_{-}+\mathbb{C}W_{+}+\mathbb{D}W_{-}= O(r^{-6})\ \ \ \text{for}\ \ \ r\ge R,
$$       
with $\mathbb{A, B}, \mathbb{C}, \mathbb{D}$ defined as same as in the above process of super-sub solution in different cases. Therefore, $\dis\frac{1}{r}(rW'_{\pm})'=O(r^{-6})$. On the other hand, notice that for each $k>1$ there exists a point $r_{k}\in (k, 2k)$ such that $W'_{\pm}(r_{k})=(W_{\pm}(2k)-W_{\pm}(k))/k=O(k^{-7})=O(r^{-7}_{k})$, then we have that
$$
|-rW'_{\pm}(r)+r_{k}W'_{\pm}(r_{k})|=\left|\int^{r_{k}}_{r}\frac{1}{r}(rW'_{\pm})'rdr\right|
\le C_2\int^{r_{k}}_{r}r^{-5}dr\le \frac{4C_{2}}{r^4},
$$  
with constant $C_2>0$.  
Together with $r_{k}W'_{\pm}(r_{k})=O(r^{-6}_{k})\to0$ as $k\to\infty$, we have that $|rW'_{\pm}(r)|\le\dis\frac{C_{2}}{r^4}$ for $r\ge R$ sufficiently large, i.e. $|W'_{\pm}(r)|\le\dis\frac{C_{2}}{r^5}$ for $r\ge R$ sufficiently large. Hence, we deduce that $|f'_{\pm}(r)+\frac{2\tilde{a}_{\pm}}{r^3}|\le\dis\frac{C_{2}}{r^5}$ for all $r\ge R$.
\end{proof}

\section{Monotonicity}

Next we will present the proof on the monotonicity of the radial solutions. First, we define the spaces as in \cite{ABM1} : 
$$
X_{0}:=H^{1}((0,\infty);rdr),
$$
$$
\dis X_{n}:=\left\{u\in X_{0}: \int^\infty_{0}\frac{u^2}{r^2}rdr<\infty\right\},\ \ \  \|u\|^2_{X_{n}}=\int^\infty_{0}\left[(u')^{2}+u^{2}+\frac{n^2}{r^2}u^2\right]rdr .
$$
Of course the spaces $X_{n}, n\neq0$, are all equivalent, but we define them this way for notational convenience. It is not difficult to show (see \cite{AB}) that for $|n|\ge1$, $X_{n}$ is continuously embedded in the space of continuous functions on $(0,\infty)$ which vanish at $r=0$ and $r\to\infty$, and that $\mathcal{C}^{\infty}_{0}((0,\infty))$ is dense in $X_{1}$. It is possible to define a global variational framework for the equivariant problems in affine spaces based on $X_{n_{+}}, X_{n_{-}}$ to prove existence of solutions. The energy is the same as in \eqref{radenergy}, except it must be "renormalized" to prevent divergence of the $\frac{n^2_{n_{\pm}}}{r^2}$ term at infinity. Here we are only interested in the (formal) second variation of this renormalized energy, 
\be\label{radvariation}
\left.
\begin{array}{l}
\dis D^{2}E_{n_{+},n_{-}}(f_{+},f_{-})[u_{+},u_{-}]:=\int^{\infty}_{0}\left[(u'_{+})^{2}+(u'_{-})^{2}+\frac{n^2_{+}}{r^2}u^2_{+}+\frac{n^2_{-}}{r^2}u^2_{-} \right]rdr\\
\ \ \ \ \ \ \ \ \dis +\int^{\infty}_{0}\left\{ [A_{+}(f^2_{+}-t^2_{+})+B(f^2_{-}-t^2_{-})]u^2_{+}+ [A_{-}(f^2_{-}-t^2_{-})+B(f^2_{+}-t^2_{+})]u^2_{-}\right\}rdr\\
\ \ \ \ \ \ \ \ \dis +\int^{\infty}_{0} 2(A_{+}f^2_{+}u^2_{+}+A_{-}f^2_{-}u^2_{-}+2Bf_{+}f_{-}u_{+}u_{-}) rdr,
\end{array}
\right.
\ee
defined for $[u_{+},u_{-}]\in X_{n_{+}}\times X_{n_{-}}$. 

We have the following fact about radial solutions:
\begin{lemma}
For any $n_{\pm}\in \mathbb{ Z}$, if $[f_{+},\ f_{-}]$ is the (unique) radial solution of \eqref{radeq}, 
$$
D^{2}E_{n_{+},n_{-}}(f_{+},f_{-})[u_{+},u_{-}]>0 \ \ \   \text{for\ \  all\ \  } [u_{+},u_{-}]\in X_{n_{+}}\times X_{n_{-}}\setminus \{[0,0]\} .
$$
\end{lemma}

In other words, the radial solutions are non-degenerate local minimizers of the renormalized energy. An analogous statement for the Ginzburg-Landau equation with magnetic field was derived in \cite{ABG}, and this observation then became the main step in the proof of uniqueness of equivariant solutions proved there.
\begin{proof}
We follow \cite{ABG}, and note that
$$
\dis f^2(r)\left[\left(\frac{u(r)}{f(r)}\right)'\right]^{2}=(u')^{2}-2\frac{uu'f'}{f}+u^{2}\frac{(f')^2}{f^2}=(u')^2-\left(\frac{u^2}{f}\right)'f'\ .
$$
Let $u_{\pm}\in\mathcal{C}^{\infty}_{0}((0,\infty))$ (if $n_{\pm}=0$, take $u_{\pm}\in\mathcal{C}^{\infty}_{0}([0,\infty))$ instead). Then $[\frac{u^2_{+}}{f_{+}},\frac{u^2_{-}}{f_{-}}]$ gives an admissible test function in the weak form of  \eqref{radeq}, 
$$
\left.
\begin{array}{l}
\dis 0=DE_{n_{+},n_{-}}(f_{+},f_{-})\left[\frac{u^2_{+}}{f_{+}},\frac{u^2_{-}}{f_{-}}\right]\\
\ \ \dis =\int^{\infty}_{0}\left[ f'_{+}\left(\frac{u^2_{+}}{f_{+}}\right)'+ f'_{-}\left(\frac{u^2_{-}}{f_{-}}\right)'+\frac{n^2_{+}}{r^2}f_{+}\frac{u^2_{+}}{f_{+}}+\frac{n^2_{-}}{r^2}f_{-}\frac{u^2_{-}}{f_{-}}\right]rdr\\
\ \ \ \ \ \ \ \dis +\int^{\infty}_{0}\bigg\{A_{+}(f^2_{+}-t^2_{+})f_{+}\frac{u^2_{+}}{f_{+}}+A_{-}(f^2_{-}-t^2_{-})f_{+}\frac{u^2_{-}}{f_{-}}\\
\ \ \ \ \ \ \ \dis +B\left[(f^2_{+}-t^2_{+})f_{-}\frac{u^2_{-}}{f_{-}}+(f^2_{-}-t^2_{-})f_{+}\frac{u^2_{+}}{f_{+}}\right] \bigg\}rdr\\
\ \ \dis=\int^{\infty}_{0}\Bigg\{(u'_{+})^{2}-f^2_{+}\left[\left(\frac{u^2_{+}}{f_{+}}\right)'\right]^{2}+(u'_{-})^{2}-f^2_{-}\left[\left(\frac{u^2_{-}}{f_{-}}\right)'\right]^{2}+\frac{n^2_{+}}{r^2}u^2_{+}+\frac{n^2_{-}}{r^2}u^2_{-} \\
\ \ \ \ \ \ \ \dis +A_{+}(f^2_{+}-t^2_{+})u^2_{+}+A_{-}(f^2_{-}-t^2_{-})u^2_{-}\\
\ \ \ \ \ \ \ \dis+B\left[(f^2_{+}-t^2_{+})u^2_{-}+(f^2_{-}-t^2_{-})u^2_{+}\right]\Bigg\}rdr\ .
\end{array}
\right.
$$
After the regrouping, we get the following useful identity:
\begin{multline*}
\dis \int^{\infty}_{0}\left\{ (u'_+)^{2}+(u'_-)^{2}+\frac{n^2_+}{r^2}u^2_{+}+\frac{n^2_-}{r^2}u^2_{-}\right.\\
\dis +A_{+}(f^2_{+}-t^2_{+})u^2_{+}+A_{-}(f^2_{-}-t^2_{-})u^2_{-}+B\left.\left[(f^2_{+}-t^2_{+})u^2_{-}+(f^2_{-}-t^2_{-})u^2_{+}\right]\right\}rdr\\
\dis =\int^{\infty}_{0}\left\{f^2_{+}\left[\left(\frac{u_{+}}{f_{+}}\right)'\right]^{2}+f^2_{-}\left[\left(\frac{u_{-}}{f_{-}}\right)'\right]^{2} \right\}rdr\ge 0.
\end{multline*}
Comparing the formula of the second variation $D^{2}E_{n_{+},n_{-}}(f_{+},f_{-})$ and the above identity, and using the fact that $\lambda_{s}>0$ is the smallest eigenvalue of the matrix $\left[\begin{array}{cc}A_+ & B \\B & A_-\end{array}\right]$, we obtain that
\begin{align}\nnn
&D^{2}E_{n_{+},n_{-}}(f_{+},f_{-})[u_{+},u_{-}] \nnn\\
&\qquad \dis =\int^{\infty}_{0}\left\{f^2_{+}\left[\left(\frac{u'_{+}}{f_+}\right)'\right]^{2}+f^2_{-}\left[\left(\frac{u'_{-}}{f_-}\right)'\right]^{2}
+2(A_{+}f^2_{+}u^2_{+}+A_{-}f^2_{-}u^2_{-}+2Bf_{+}f_{-}u_{+}u_{-})\right\}rdr\nnn\\
&\qquad \ge \dis 2\int^{\infty}_{0}(A_{+}f^2_{+}u^2_{+}+A_{-}f^2_{-}u^2_{-}+2Bf_{+}f_{-}u_{+}u_{-})rdr\nnn\\
&\qquad \ge \dis 2\lambda_{s}\int^{\infty}_{0}(f^2_{+}u^2_{+}+f^2_{-}u^2_{-})rdr\nnn \\
&\qquad \ge0\label{2ndvarinfty}, 
\end{align}
valid for all $u_{\pm}\in\mathcal{C}^{\infty}_{0}((0,\infty))$ (or, $u_{\pm}\in\mathcal{C}^{\infty}_{0}([0,\infty))$ if the respective $n_{\pm}=0$). The case of general $u_{\pm}\in X_{1}$ (or $X_{0}$, in case one of $n_{\pm}=0$) follows by density. It is clear that $D^{2}E_{n_{+},n_{-}}(f_{+},f_{-})\ge 0$. If it were zero for some $[u_{+},u_{-}]$, then we would have $f_{+}u_{+}=-f_{-}u_{-}=0$ almost everywhere. Since $f_{\pm}(r)>0$ for $r>0$, we conclude that $D^{2}E_{n_{+},n_{-}}(f_{+},f_{-})>0$ as claimed.
\end{proof}

We are now ready for the proof of Theorem~\ref{monotone} from the introduction.

%
%

\begin{proof}[Proof of Theorem~\ref{monotone}]
Parts (i) and (ii) follow from the nondegeneracy of $D^2E_{n_+,n_-}$ above; part (iii) will be proven in a different manner, using the asymptotics and a compactness argument.

Let $u_{\pm}(r):=f'_{\pm}(r)$. Differentiating \eqref{radeq}, we get
\begin{multline*}
\dis -f'''_{\pm}-\frac1r f''_{\pm}+\frac{n^2_{\pm}+1}{r^2}f'_{\pm}+[A_{\pm}(f^2_{\pm}-t^2_{\pm})+B(f^2_{\mp}-t^2_{\mp})]f'_{\pm}\\
\dis +2A_{\pm}f^2_{\pm}f'_{\pm}+2Bf_{\pm}f_{\mp}f'_{\mp}-\frac{2n^2_{\pm}}{r^3}f_{\pm}=0,
\end{multline*}
i.e. 
\begin{multline}\label{upmeq}
\dis -u''_{\pm}-\frac1r u'_{\pm}+\frac{n^2_{\pm}+1}{r^2}u_{\pm}+[A_{\pm}(f^2_{\pm}-t^2_{\pm})+B(f^2_{\mp}-t^2_{\mp})]u_{\pm}\\
\dis +2A_{\pm}f^2_{\pm}u_{\pm}+2Bf_{\pm}f_{\mp}u_{\mp}-\frac{2n^2_{\pm}}{r^3}f_{\pm}=0.
\end{multline}

Now define $v_{\pm}=\min\{0,u_{\pm}\}\le0, w_{\pm}=\max\{0, u_{\pm}\}\ge0$, then $u_{\pm}=v_{\pm}+w_{\pm}$. 
For part (i), assume $B<0$.  We multiply the respective equation in \eqref{upmeq} by $v_{\pm}$, use the facts $v_{+}w_{+}=0$ and $v_{-}w_{-}=0$ and integrate by parts. By \eqref{fpstive} and \eqref{fnear0}, if $n_{\pm}\ge1$, $u_{\pm}(r)=f'_{\pm}(r)>0$ in some neighborhood $r\in(0,\delta)$. Thus, in case $n_{\pm}\ge1$, $v_{\pm}$ is supported away from $r=0$. By the proof of Theorem~\ref{asymptotics} we may conclude that $v_{\pm}\in X_{n_{\pm}}$. Moreover, 
$$
\dis \int^\infty_{0}v_{\pm}\left(u''_{\pm}+\frac1r u'_{\pm}\right)rdr=-\int^\infty_{0}(v'_{\pm})^{2}rdr,
$$
with no boundary terms. In case $n_{\pm}=0$, we have $u_{\pm}\in X_{0}$ by the regularity of solutions, and $u_{\pm}(0)=f'_{\pm}(0)=0$. The integration by parts formula above again holds with no boundary terms in this case as well. Therefore, with all above facts, we have the following equations
\begin{multline*}
0=\dis\int^\infty_{0}\left\{(v'_{\pm})^{2}+\frac{n^2_{\pm}+1}{r^2}v^2_{\pm}+[A_{\pm}(f^2_{\pm}-t^2_{\pm})+B(f^2_{\mp}-t^2_{\mp})]v^2_{\pm}\right.\\
\dis \left.+2A_{\pm}f^2_{\pm}v^2_{\pm}+2Bf_{\pm}f_{\mp}u_{\mp}v_{\pm}-\frac{2n^2_{\pm}}{r^3}f_{\pm}v_{\pm}\right\}rdr\ .
\end{multline*}
Then, use the facts $u_{\pm}=v_{\pm}+w_{\pm}$ and add above two equations together
\begin{multline*}
0=\dis\int^\infty_{0}\left\{(v'_{+})^{2}+(v'_{-})^{2}+\frac{n^2_{+}}{r^2}v^2_{+}+\frac{n^2_{-}}{r^2}v^2_{-}\right.\\
\dis +[A_{+}(f^2_{+}-t^2_{+})+B(f^2_{-}-t^2_{-})]v^2_{+}+[A_{-}(f^2_{-}-t^2_{-})+B(f^2_{+}-t^2_{+})]v^2_{-}\\
\dis +2(A_{+}f^2_{+}v^2_{+}+A_{-}f^2_{-}v^2_{-})+4Bf_{+}f_{-}v_{+}v_{-}+2Bf_{+}f_{-}(w_{-}v_{+}+w_{+}v_{-})\\
\dis \left.-\frac{2n^2_{+}}{r^3}f_{+}v_{+}-\frac{2n^2_{-}}{r^3}f_{-}v_{-}\right\}rdr\ .
\end{multline*}
Compare with the formula of the second variation $D^{2}E_{n_{+},n_{-}}(f_{+},f_{-})$ in \eqref{2ndvarinfty}, we obtain 
\begin{align*}
0&=\dis D^{2}E_{n_{+},n_{-}}(f_{+},f_{-})[v_{+},v_{-}]+2B\int^\infty_{0}f_{+}f_{-}(w_{-}v_{+}+w_{+}v_{-})rdr\\
&\quad \dis+\int^\infty_{0}\frac{1}{r^2} (v^2_{+}+v^2_{-})rdr-2\int^\infty_{0}\frac{1}{r^3} (n^2_{+}f_{+}v_{+}+n^2_{-}f_{-}v_{-})rdr\ .
\end{align*}
Since $n^2_{+}f_{+}v_{+}$, $n^2_{-}f_{-}v_{-}$, $w_{-}v_{+}$ and $w_{+}v_{-}$ are all negative, together with $B<0$, we get that
\begin{align*}
0&\le D^{2}E_{n_{+},n_{-}}(f_{+},f_{-})[v_{+},v_{-}] \\
&=\dis- \int^\infty_{0}\frac{1}{r^{2}} (v^2_{+}+v^2_{-})rdr+2\int^\infty_{0}\frac{1}{r^3} (n^2_{+}f_{+}v_{+}+n^2_{-}f_{-}v_{-})rdr\\
&\hskip .8cm \dis -2B\int^\infty_{0}f_{+}f_{-}(w_{-}v_{+}+w_{+}v_{-})rdr\\
&\le \dis- \int^\infty_{0}\frac{1}{r^{2}} (v^2_{+}+v^2_{-})rdr\\
&<0,
\end{align*}
which is a contradiction unless $v_{\pm}\equiv0$, i.e. unless $f'_{\pm}(r)\ge0$ for all $r>0$. This proves (i).

To prove (ii), assume $B>0$ and $n_{+}\ge1$ and $n_{-}=0$. This time we multiply the equation of $u_+$ by $v_+$ and the equation of $u_-$ by $v_-$, and integrate by parts again. Just as in the previous case, $w_{-}\in X_{0}$, and the boundary term in the integration will all vanish. We obtain that:
\begin{multline*}
0=\dis \int^\infty_{0}\left\{(v'_{+})^{2}+\frac{n^2_{+}+1}{r^2}v^2_{+}+[A_{+}(f^2_{+}-t^2_{+})+B(f^2_{-}-t^2_{-})]v^2_{+}\right.\\
\dis \left.+2A_{+}f^2_{+}v^2_{+}+2Bf_{+}f_{-}u_{-}v_{+}-2\frac{n^2_{+}}{r^3}f_{+}v_{+}\right\}rdr\ ,
\end{multline*}
and
\begin{multline*}
0=\dis \int^\infty_{0}\left\{(w'_{-})^{2}+\frac{n^2_{-}+1}{r^2}w^2_{-}+[A_{-}(f^2_{-}-t^2_{-})+B(f^2_{+}-t^2_{+})]w^2_{-}\right.\\
\dis \left.+2A_{-}f^2_{-}w^2_{-}+2Bf_{-}f_{+}u_{+}w_{-}-2\frac{n^2_{-}}{r^3}f_{-}w_{-}\right\}rdr\ .
\end{multline*}
As in the first case, we add the above two equations and compare to the formula of the second variation $D^{2}E_{n_{+},n_{-}}(f_{+},f_{-})$ in \eqref{2ndvarinfty}, we get 
\begin{align*}
0&\le D^{2}E_{n_{+},n_{-}}(f_{+},f_{-})[v_{+},w_{-}]\\
&=\dis -\int^\infty_{0}\frac{1}{r^2}(v^2_{+}+w^2_{-})rdr-2B\int^\infty_{0}f_{+}f_{-}(v_{-}v_{+}+w_{+}w_{-})rdr\\
&\qquad \dis +2\int^\infty_{0}\left(\frac{n^2_{+}}{r^3}f_{+}v_{+}+\frac{n^2_{-}}{r^3}f_{-}w_{-}\right)rdr\\
&\le \dis -\int^{\infty}_{0} \frac{1}{r^2}(v^2_{+}+w^2_{-})rdr\\
&<0\ ,
\end{align*}
which is a contradiction unless $v_{+}\equiv0,\ w_{-}\equiv0$, i.e. $f'_{+}(r)\ge0$ and $f'_{-}(r)\le0$ for all $r>0$. This verifies statement (ii).

\medskip

It remains to prove part (iii) of Theorem~\ref{monotone}.
Denote $f_{\pm}(r; B)$ the solution of \eqref{radeqR} with coefficient $B$.
We recall that from \eqref{fnear0} of Theorem~\ref{lemunique}, $f_\pm(r;B)\sim r^{n_\pm}$ for $r$ near $0$, that is, each component vanishes exactly to degree $n_\pm$ at $r=0$,
\begin{equation}\label{f00}
f^{(m)}_\pm(0;B)=0, \quad m=0,\dots, n_\pm-1, \qquad 
   f^{(n_\pm)}_{\pm}(0;B)>0.
\end{equation}
When $B=0$ the system decouples, and each of $f_{0,\pm}=f_{\pm}(\cdot;0)$ solves a rescaled equation for the standard Ginzburg-Landau vortices,
\begin{equation}\label{GLrad}
\left.
\begin{gathered}
 \dis -f''_{0,+}-\frac{1}{r}f'_{0,+}+\frac{n^2_{+}}{r^2}f_{0,+}+
 A_{+}(f^2_{0,+}-t^2_{+})f_{0,+}=0, \ \ {\text{for}}\ r\in (0,\ \infty), \\
 \dis -f''_{0,-}-\frac{1}{r}f'_{0,-}+\frac{n^2_{-}}{r^2}f_{0,-}+ A_{-}(f^2_{0,-}-t^2_{-})f_{0,-}=0, \ \ {\text{for}}\ r\in (0,\ \infty), \\
 f_{0,\pm}(0)=0, \quad f_{0,\pm}(r)\to t_\pm \ \text{as $r\to\infty$},\quad
f_{0,\pm}(r)\ge 0\ \ \text{for\ \ all}\ \ r\in [0,\ \infty).
\end{gathered}
\right\}
\end{equation}
 For these solutions, it is well-known that (see \cite{HH})
\begin{equation}\label{glmonotone}
f'_{0,\pm}(r)>0\quad\text{ for $\forall r>0$}.
\end{equation}
Now choose $\tilde B_{0}>0$ for which the leading-order term in the asymptotic expansion (given in \eqref{apm}) $a_{\pm}=a_\pm(\tilde B_0)<0$.
From the formula \eqref{apm}, it is clear that $a_\pm=a_\pm(B)<0$ for all $B\in [0,\tilde B_0]$.
 By the asymptotic estimate \eqref{f'pmasybound}, there exists $R>0$, which may be chosen uniformly for $B\in [0,\tilde B_0]$, such that
\begin{equation}\label{f'(r;B)}
f'_{\pm}(r; B)>0\ \ \ \ \text{for}\ \ \ \forall r\ge R, \forall B\in[0, B_{0}],
\end{equation}
here $f'_{\pm}(r; B)$ depends on both $R$ and parameter $B$.

We next claim that $f_\pm(r;B)\to f_\pm(r;0)=f_{0,\pm}(r)$ in $\mathcal{C}^k([0,R])$ as $B\to 0$, for any $k\ge 0$.  Indeed, from Proposition~\ref{lemunique}
$\psi_{B,\pm}(x)=f_{\pm}(r;B)e^{i n_\pm\theta}$ is an entire solution to \eqref{eqns} in $\RR^2$.  As the right-hand sides of \eqref{eqns} are $L^\infty$ bounded, uniformly for $B\in [0,B_0]$, by standard elliptic estimates it follows that the solutions are bounded in $\mathcal{C}^{1,\alpha}_{loc}$ for any $\alpha\in (0,1)$.  Moreover, the right-hand sides of \eqref{eqns} are polynomials, and so by a bootstrap argument we obtain the stronger conclusion that $\psi_{B,\pm}$ are uniformly bounded in $\mathcal{C}^k_{loc}$ for any $k\ge 0$.  In particular, for any sequence $B_n\to 0$ there is a subsequence such that $\psi_{B_n,\pm}\to \psi_{0,\pm}$ in $\mathcal{C}^k$-norm, for any $k\ge 2$.  By the form of the solutions, we have $\psi_{0,\pm}=f_{0,\pm}(r) e^{in_\pm \theta}$, and
by passing to the limit in the equations we recognize $f_{0,\pm}=f_\pm(\cdot;0)$ as the (unique) solutions to the decoupled Ginzburg-Landau system \eqref{GLrad} above.  By the uniqueness of the limit, we conclude that $f_\pm(r;B)\to f_\pm(r;0)=f_{0,\pm}(r)$ in $\mathcal{C}^k([0,R])$ as $B\to 0$ (that is, along any sequence $B_n\to 0$), as claimed.

We now complete the proof of the statement (iii) by contradiction:  suppose there exist sequences $B_{k}\to 0$ and $r_{k}\in (0, \infty)$ such that $f'_{\pm}(r_{k}; B_{k})\le 0$. By \eqref{f'(r;B)}, we have that $0<r_{k}\le R$, so there must exist a subsequence $r_{k_{j}}$ and $r_0\in [0,\infty)$ with $r_{k_{j}}\to r_{0}$.  We claim that $r_0=0$.  Indeed, if not, by the $\mathcal{C}^1([0,R])$ convergence proven above,
$$
f'_{0,\pm}(r_0)=f'_{\pm}(r_{0}; 0)=\dis\lim_{j\to\infty}f'_{\pm}(r_{k_{j}}; B_{k_{j}})\le 0,
$$
which contradicts \eqref{glmonotone}.  Thus $r_0=0$.  

Since $f'_\pm(r_{k_j};B_{k_j})=0$, by applying Rolle's theorem $(n_\pm-1)$ times we may conclude that there exist points 
$\tilde r_{k_j}\in (0,r_{k_j})$ so that $f^{(n_\pm)}(\tilde r_{k_j})=0$.  By the $\mathcal{C}^{n_\pm}([0,R])$ convergence, we then conclude that
$$  f^{(n_\pm)}_{0,\pm}(0) = 
\lim_{j\to\infty} f^{(n_\pm)}_{\pm}(\tilde r_{k_j};B_{k_j}) =0, $$
which contradicts \eqref{f00} for $B=0$.
Thus we conclude that $f'(r;B)>0$ for all $r\in (0,\infty)$ and for all $B\in [0, B_0]$, as desired.
\end{proof}

\end{document}